\documentclass[preprint,11pt]{article}

\usepackage{braket,amsfonts}
\usepackage{amssymb}
\usepackage{amsthm}
\usepackage{array}
\usepackage{url}

\usepackage{pgfplots}
\usepackage{hyperref}
\usepackage{authblk}
\usepackage[a4paper, total={7in, 9in}]{geometry}
\usepackage{graphicx}
\newtheorem{prop}{Proposition}
\newtheorem{theorem}{Theorem}
\usepackage{amsopn}

\tikzset{
  treenode/.style = {shape=rectangle, rounded corners,
                     draw, align=center,
                     top color=white, bottom color=blue!20},
  root/.style     = {treenode, font=\Large, bottom color=red!30},
  env/.style      = {treenode, font=\ttfamily\normalsize},
  dummy/.style    = {circle,draw}
}
\usetikzlibrary{positioning,calc}

\DeclareMathOperator*{\argmin}{arg\,min}
\usepackage{algorithm}

\usepackage{algpseudocode}
\usepackage{mathtools}

\newtheorem{remark}{Remark}[section]
\newcommand{\R}{\mathbb{R}}

\usepackage{xspace}
\usepackage{bold-extra}
\usepackage[most]{tcolorbox}

\begin{document}


\title{Statistical Proper Orthogonal Decomposition for model reduction in feedback control}

\author[1]{Sergey Dolgov}
\author[2]{Dante Kalise}
\author[3]{Luca Saluzzi}
\affil[1]{Department of Mathematical Sciences, University of Bath, United Kingdom. 
	
	\texttt{e-mail:s.dolgov@bath.ac.uk}}
\affil[2]{Department of Mathematics, Imperial College London, United Kingdom. \texttt{e-mail:dkaliseb@ic.ac.uk}}
\affil[3]{Department of Mathematics, Scuola Normale Superiore, Pisa, Italy. \texttt{e-mail:luca.saluzzi@sns.it}}

\maketitle
\begin{abstract}
Feedback control synthesis for nonlinear, parameter-dependent fluid flow control problems is considered. The optimal feedback law requires the solution of the Hamilton-Jacobi-Bellman (HJB) PDE suffering the curse of dimensionality. This is mitigated by Model Order Reduction (MOR) techniques, where the system is projected onto a lower-dimensional subspace, over which the feedback synthesis becomes feasible.
However, existing MOR methods assume at least one relaxation of generality, that is, the system should be linear, or stable, or deterministic. 

We propose a MOR method called Statistical POD (SPOD), which is inspired by the Proper Orthogonal Decomposition (POD), but extends to more general systems. Random samples of the original dynamical system are drawn, treating time and initial condition as random variables similarly to possible parameters in the model, and employing a stabilizing closed-loop control. The reduced subspace is chosen to minimize the empirical risk, which is shown to estimate the expected risk of the MOR solution with respect to the distribution of all possible outcomes of the controlled system.
This reduced model is then used to compute a surrogate of the feedback control function in the Tensor Train (TT) format that is computationally fast to evaluate online.
Using unstable Burgers' and Navier-Stokes equations, it is shown that the SPOD control is more accurate than Linear Quadratic Regulator or optimal control derived from a model reduced onto the standard POD basis, and faster than the direct optimal control of the original system.
\end{abstract}

\section{Introduction}
\label{sec:intro}
Nonlinear dynamics arising in fluid flow play a vital role in numerous industrial applications, ranging from aerospace engineering to chemical processing. Efficient management and control of these systems under random state fluctuations or different parameters is crucial to ensure optimal performance, stability, and safety. 
One key approach to achieving these objectives is feedback control. This involves continuous monitoring of the system's behavior and applying corrective actions to minimize deviations from a reference state. In the context of fluid flow problems, feedback control enables engineers to actively manipulate various system parameters such as flow rates, pressures, temperatures, and concentrations to stabilize the dynamics.
In the case of linear dynamics and quadratic cost functions, a feedback controller that is relatively cheap computationally is the Linear Quadratic Regulator (LQR), which reduces to the solution of a matrix Algebraic Riccati equation \cite{Kirsten_Simoncini_2020,BBKS_2020}.
However, the complexity and nonlinearity inherent in fluid dynamics pose significant challenges in the design and implementation of feedback control strategies. Unlike other fields where linear models and simplified assumptions may be sufficient, fluid systems often exhibit intricate behaviors, including turbulence, flow separation, and time-dependent phenomena. Consequently, developing effective feedback control techniques for fluid problems turns out to be a very challenging task. 

Potentially, the optimal feedback control can be computed by solving the Hamilton-Jacobi-Bellman (HJB) nonlinear partial differential equation (PDE). 
However, the HJB equation is usually of hyperbolic nature, and posed on the state space of the dynamical system, which can be very high-dimensional. 
The latter is typical when the system arises from a discretization of a PDE, such as the Navier-Stokes equation in fluid dynamics.
Moderately high-dimensional HJB equations have been tackled with various methods:
max-plus algebra \cite{Akian_Gaubert_Lakhoua_2008,maxplusdarbon}, sparse grids and polynomials \cite{GK16,AKK21,KVW23}, tree-structure algorithms \cite{alla2019high,falcone2023approximation}, deep neural networks \cite{Han_Jentzen_E_2018,Darbon_Langlois_Meng_2020,Kunisch_Walter_2021,sympocnet,Zhou_2021,Onken2021,ruthotto2020machine,ABK21}, low-rank tensor decompositions \cite{sallandphd,richter2021solving,DKK21} and kernel interpolation techniques \cite{alla2021hjb,ehring2023hermite}.
However, direct solution of the HJB equation for thousands of dimensions (which is typical in Finite Element discretizations) remains out of reach.

The PDE formulation of the HJB equation can be 
bypassed by a data-driven approach, where one computes the value function and/or control at each of the given (usually random) states in a closed-loop fashion, and finds an approximate feedback control in some ansatz by solving a regression problem by minimizing an empirical risk of given samples of the control.
Such regression in a sparse polynomial basis was considered e.g. in 
\cite{Kang_Wilcox_2017,Kang_Wilcox_2015,Azmi:2021}, a tensor format regression was employed in \cite{oss-hjbt-2021,sallandphd,dolgov2022data}, and a neural network was trained in \cite{Nakamura_Zimmerer_2021,Nakamura_Zimmerer_2021b,bensnn,onken2021neural}.
However, the dimension of the sought control function is still that of the state space, requiring a sheer number of unknowns in the approximation ansatz and training data.

Fortunately, not all states of the system are important or reachable, especially in the controlled regime.
Therefore, a promising approach to tackle the curse of dimensionality is the Model Order Reduction (MOR).
Here, the dynamical system is projected onto some precomputed basis of lower dimension, and the HJB equation (or a data-driven variant thereof) is solved for the reduced system.
The key question now is the algorithm to compute such a basis.

One of the simplest options is called Proper Orthogonal Decomposition (POD). This method collects system states ("snapshots") at certain time points, and finds a basis that minimizes the total projection error of these snapshots onto the basis.
In the context of feedback control, the POD method was used to deliver a reduced system (and hence a reduced HJB equation) in \cite{kunisch2004hjb,kunisch2005pod}.
However, the snapshots produced from an uncontrolled deterministic system may be inaccurate, and even misleading for the controlled system at a different realisation of random parameters or initial conditions.
For instance, it may be simply impossible to collect snapshots from an uncontrolled system that exhibits a finite-time blowup.

A state of the art method to identify a basis that approximates well only controllable and observable states is balanced truncation (BT) \cite{Moore-BT-1981}.
Note that the original formulation of BT still uses the uncontrolled system, and thus requires its stability, as well as linearity.
This problem was mitigated by the closed-loop balanced truncation \cite{Wortelboer} that incorporates the LQR, stabilizing the system matrix.
The linearity assumption aside, the reduction of the closed-loop instead of uncontrolled system will be the first ingredient of our paper.

Since BT uses the linear structure of the dynamical system quite explicitly, its generalization to nonlinear systems is difficult, and still limited to relatively simple cases.
For example, a BT for bilinear systems was introduced in \cite{BenD11},
lifting of some nonlinear systems to quadratic-bilinear systems of larger dimension was proposed in \cite{Kramer-lifted-BT-2022},
and an $\mathcal{H}_{\infty}$-balanced truncation for differential-algebraic systems, linearizing the system around one state (the stationary solution), was presented in \cite{Heiland-Hinf-BT-2022}. A more general approach to nonlinear BT has been recently developed in \cite{kramer2022nonlinear,kramer2023nonlinear} however, the nonlinear reduction also requires the solution of high-dimensional HJB type equations, which is precisely the computational limitation we try to avoid in this paper. As we will see in numerical examples, linearization around one particular state may be inaccurate for more general nonlinearities and random parameters, such as the Navier-Stokes equation with random inflow and boundary control.

Random parameters (such as random coefficients or initial conditions) are particularly difficult but relevant scenarios motivating the feedback control, which needs to be robust with respect to this uncertainty.
If the random variables enter the model (bi)linearly, they can also be lifted to apply e.g. bilinear BT \cite{Redmann-BT-stoch-2015}.
Existence of a quadratic-bilinear structure suitable for BT is not clear in general.
On the other hand, POD methods require only that random realizations of the model can be sampled: in this case, the minimization of the projection error in POD can be straightforwardly extended to the minimization of an empirical risk calculated using training samples of all random variables in the model. A deep-learning/POD approach for parameter-dependent nonlinear PDEs has been developed in \cite{FDM21}.

Empirical balanced truncation using snapshots to approximate the Gramians emerged soon after the first BT papers \cite{Lall-snapshot-BT-1999}, and many more followed.
Similarly to \cite{Lall-snapshot-BT-1999}, \cite{rowley-balanced-pod-2005} proposes a balanced POD to compute the Gramians from numerical simulations of state responses to unit impulses.
Using similar ideas, \cite{Willcox-BT-POD-2002} shows that the POD kernel is an approximation to the controllability Gramian.
This is an important observation for this paper: a feedback control applies usually to the fully observed state, in other terms the observation matrix is the identity, which drops the main benefit of BT compared to POD: the extra truncation of poorly observable states.

Thus, we propose a MOR method that resembles the POD to a large extent, except that the snapshots are collected at random samples of time, parameters and initial conditions.
In contrast to particularly designed inputs (such as unit impulses), this statistical approach makes the empirical risk a convergent estimate of the expected risk.
This method is motivated by Likelihood Informed Subspace \cite{MCMC:CLM_2016} and Active Subspace \cite{DimRedu:CDW_2014,zahm2020gradient} methods in statistics, where the reduced basis is derived from an empirical mean Hessian of the log-likelihood, or Gramian of the forward model.
Here we propose a similar approach in the feedback control context.
Stabilization is achieved by collecting the snapshots from a closed-loop system similarly to \cite{Wortelboer}, but controlled with any feasible (possibly suboptimal) regulator, such as the Pontryagin Maximum Principle \cite{Kirk2004}, or State-Dependent Riccati Equation \cite{Banks_Lewis_Tran_2007}.
The proposed method is close to \cite{Yousefi-nlBT-2004} which linearizes the system by estimating system matrices from snapshots, followed by balanced truncation,
and \cite{Olshanskii-TT-POD-2023} where the POD is extended to a low-rank tensor approximation of the full solution as a function of state, parameters and time, followed by recompression of the basis at the given parameter.
However, in our statistical POD method we neither assume nor seek any linearization, and in contrast to \cite{Olshanskii-TT-POD-2023} we do this in the optimal control setting.

Moreover, as soon as the statistical POD basis is ready, we pre-compute a low-rank tensor approximation (specifically, the Functional Tensor Train (TT) format \cite{osel-tt-2011,Marzouk-stt-2016,Gorodetsky-ctt-2019}) of the closed-loop control of the reduced model as a function of the reduced state.
This TT approximation provides the desired control in the feedback form that is fast to evaluate numerically, since it needs only a modest amount of linear algebra operations with the TT decomposition, in contrast to solving an optimal control problem from scratch.
This enables fast synthesis of a nearly-optimal control in the online regime.
A schematic of the entire procedure is shown in Figure~\ref{fig:workflow}.

\begin{figure}[h!]
\centering
\begin{tikzpicture}
\node[draw=black,rounded corners,text width=0.19\linewidth] (samples) {Sample initial condition and parameters};
\node[draw=black,rounded corners,right=0.08\linewidth of samples,text width=0.23\linewidth] (ode)  {Solve closed-loop full systems};
\node[draw=black,rounded corners,right=0.08\linewidth  of ode,text width=0.22\linewidth] (svd) {trunc. SVD of snapshot matrix};
\node[draw=black,rounded corners,below=0.1\linewidth of svd,text width=0.23\linewidth,minimum height=0.1\linewidth] (redsamples) {Solve closed-loop reduced systems};
\node[draw=black,rounded corners,left=0.20\linewidth of redsamples,text width=0.20\linewidth,minimum height=0.1\linewidth] (cross) {TT-Cross};
\node[left=0.1\linewidth of cross] (offline) {\textbf{offline}};
\node[draw=black,rounded corners,below=0.15\linewidth of redsamples,text width=0.12\linewidth] (fullon) {Online system};
\node[draw=black,rounded corners,left=0.3\linewidth of fullon,text width=0.17\linewidth] (interp) {Interpolate TT control};
\node[below=0.20\linewidth of offline] (online) {\textbf{online}};

\draw[->,line width=1pt] ($(samples.east)+(0,0.02\linewidth)$) -- ($(ode.west)+(0,0.02\linewidth)$);
\draw[->,line width=1pt] ($(samples.east)-(0,0.02\linewidth)$) -- ($(ode.west)-(0,0.02\linewidth)$);
\draw[->,line width=1pt] (samples.east) -- (ode.west);
\draw[->,line width=3pt] (ode.east) -- (svd.west);
\draw[->,dashed,line width=2pt] (svd.south east) to[out=-45,in=45] node[midway,left] {basis} (redsamples.north east);
\draw[->,line width=1pt] ($(cross.east)+(0,0.02\linewidth)$) to[out=45,in=135,looseness=1] node[above,midway] {{\footnotesize state samples}} ($(redsamples.west)+(0,0.02\linewidth)$);
\draw[->,line width=1pt] ($(redsamples.west)-(0,0.02\linewidth)$) to[out=-135,in=-45,looseness=1] node[below,midway] {{\footnotesize control samples}} ($(cross.east)-(0,0.02\linewidth)$);

\node[circle,draw=black,inner sep=2pt,text width=0.08\linewidth] (project) at ($0.5*(fullon.north west)+0.5*(interp.north east)+(0,0.02\linewidth)$) {\footnotesize project};

\draw[->,line width=1pt] ($(fullon.west)+(0,0.02\linewidth)$) to[out=135,in=0,looseness=1] node[below,midway,rotate=-25] {{\footnotesize state}} (project.east);
\draw[->,line width=1pt] (project.west) to[out=180,in=45,looseness=1] node[below,midway,rotate=25] {{\footnotesize $\substack{\mbox{red.}\\\mbox{state}}$}} ($(interp.east)+(0,0.02\linewidth)$);
\draw[->,line width=1pt] ($(interp.east)-(0,0.02\linewidth)$) to[out=-45,in=-135,looseness=1] node[above,midway] {{\footnotesize control}} ($(fullon.west)-(0,0.02\linewidth)$);

\draw[->,dashed,line width=2pt] (svd.south east) to[out=-45,in=10,looseness=1.7]  (project.north east);
\draw[->,dashed,line width=2pt] (cross.south west) to[out=-135,in=135] node [midway,right] {TT format} (interp.north west);
\end{tikzpicture}
\caption{Workflow of the statistical POD and TT controller. In the offline stage, we draw (a modest number of) random samples of the initial condition and parameters, solve the full closed-loop systems on those parameters, and collect all snapshots into one matrix. Computing the truncated SVD of this matrix delivers a reduced basis minimizing the empirical risk. The TT-Cross algorithm is run to approximate the optimal control of the reduced system in the feedback form in the TT format. The TT-Cross draws adaptive samples of the reduced state, and assimilates the corresponding control values. In the online stage, the system state is measured and projected onto the reduced basis. The TT format of the control is interpolated on this reduced state and fed back into the system.}
\label{fig:workflow}
\end{figure}
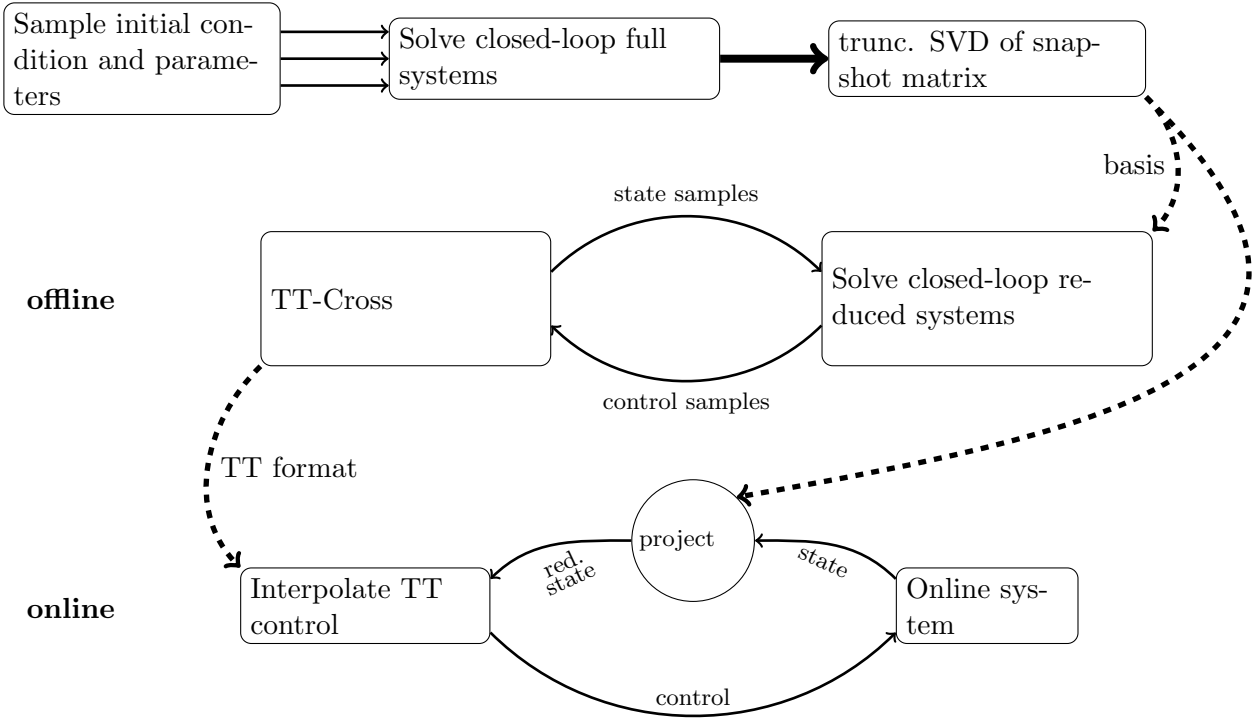
 
The rest of the paper is organized as follows. In Section 2 we introduce relevant background material on optimal control and tensor train approximation. In Section 3 we introduce our statistical POD approach, and present related error estimates in Section 4. The proposed methodology is first assessed in Section 5 using Burgers' equation, to conclude with a full application to fluid flow control in Navier-Stokes in Section 6.
\section{Background}
\subsection{Parameter-dependent optimal control problems}
We are interested in a class of parameter-dependent optimal control problems where a parameter $\mu\in \R^{M}$ accounts for uncertainties in initial/boundary conditions of the control system propagating along state and control trajectories. 
To introduce $\mu$ we assume a finite dimensional noise. 
Namely, given a complete probability space $(\Omega, \mathcal{F}, \mathbb{P})$ for random realisations of the system $\omega\in\Omega$, the $\sigma$-algebra of events $\mathcal{F}$ and the probability measure $\mathbb{P}$, 
we assume that any random variable $X(\omega)$ can be expressed as a deterministic function $X(\omega)=x(\mu(\omega))$ of the random vector $\mu(\omega)$ with a product density function $\pi(\mu) = \pi(\mu_1) \cdots \pi(\mu_M),$ and expectations with respect to $\mathbb{P}$ can be computed as integrals weighted with $\pi$,
 $
 \mathbb{E}[X] = \int_{\mathbb{R}^M} x(\mu) \pi(\mu) d\mu.
 $
Now the system dynamics can be described as a parameter-dependent controlled evolution equation:
\begin{equation}\label{eq}
\left\{ \begin{array}{l}
\frac{d}{dt}y(t;\mu)=f(y(t;\mu),u(t;\mu),\mu), \;\; t\in(0,T],\, \mu\in \R^M,\\
y(0;\mu)=x(\mu) \in\mathbb{R}^d,  
\end{array} \right.
\end{equation}
where we denote by $y:[0,T] \times \R^M\rightarrow\R^d$ the state of the system, by $u:[0,T] \times \R^M \rightarrow\R^m$ the control signal, and by $\mathcal{U}=L^\infty ([0,T] \times \R^M;U_{ad}), U_{ad}\subseteq \R^m$, the set of admissible controls. 

Given the control system \eqref{eq}, we are concerned with the synthesis of a feedback control law $u(y(t;\mu);\mu)$, that is, a control law that primarily depends on the current state of the system $y(t;\mu)$. We design such a control law by minimizing a finite horizon cost functional of the form
\begin{equation}\label{cost_finite}
 J_T(u;x,\mu):=\int_0^{T} L(y(t;\mu),u(t;\mu))\,dt \,, \; \mu \in \R^M, u \in \mathcal{U},
\end{equation}
where $L:\mathbb{R}^d \times \mathbb{R}^m \rightarrow \mathbb{R}$ is a suitable running cost. Note that the dependence on the cost is with respect to $x$, the initial condition, which generates the trajectory $y(t;\mu)$  For a given parameter realization $\mu$ and initial condition $x(\mu)$, the optimal control signal is given by the solution of 
\begin{equation}
    \underset{u(\cdot)\in\mathcal{U}}{\inf} J_T(u;x,\mu)\,,\qquad\text{subject to \eqref{eq}}\,.
    \label{VF_finite}
\end{equation}

For a fixed $\mu$, we obtain deterministic dynamics, for which optimality conditions for the dynamic optimization problem \eqref{VF_finite} are given by Pontryiagin's Maximum Principle (PMP)
\begin{align}\label{pmp}
\frac{d}{dt}y(t;\mu)&=f(y(;\mu),u(t;\mu),\mu), \\
y(0;\mu)&=x(\mu),\\
-\frac{d}{dt} p(t;\mu)&=  \nabla_{y} f(y(t;\mu),u(t;\mu),\mu)^\top p(t;\mu)+\nabla_{y}L(y(t;\mu),u(t;\mu)), \\
p(T;\mu)&=0 , \\
u(t;\mu)&=\underset{w \in U_{ad}}{\argmin}  \{L(y(t;\mu),w) +f(y(t;\mu),w,\mu)^{\top} p(t;\mu)  \}\,,
\label{pmp2}
\end{align}
where  $p:[0,T] \times \R^M\rightarrow\R^d$ denotes the adjoint variable. The PMP system can be solved using a reduced gradient method (see $e.g.$ \cite{AKK21} for more details), which requires an initial guess for the optimal control signal $u(t;\mu)$. In Section \ref{eff_snap}, we will explain how the construction of reduced basis can improve the choice for the initial guess and accelerate the entire algorithm.

\subsection{Sub-optimal control laws using the State-Dependent Riccati Equation}
\label{sec:sdre}

The State-Dependent Riccati Equation (SDRE) is an effective technique for feedback stabilization of nonlinear dynamics \cite{ccimen2008state,allasdre}. The method is based on the sequential solution of linear-quadratic control problems arising from sequential linearization of the dynamics along a trajectory. 

Given a realization of $\mu$, we consider an unconstrained ($U_{ad}=\R^m$) infinite-horizon quadratic cost functional
\begin{equation}
    J_{\infty}(u;x,\mu) = \int\limits_0^{+\infty} y(t;\mu)^\top Q y(t;\mu) + u(t)^\top R u(t)\, dt \,,
\end{equation}
with $Q\in\R^{d\times d}, Q\succeq 0\,,\; R\in\R^{m\times m}, R\succ 0\,,$ and system dynamics expressed in semilinear form
\begin{align}
    \frac{d}{dt}y(t;\mu) & = A(y(t;\mu);\mu) y(t;\mu) +B(y(t;\mu);\mu) u(t) \\
    y(0;\mu) & = x(\mu)\,.
\end{align}
In the the case of linear dynamics in the state, $A(y(t;\mu);\mu) =A(\mu)\in\R^{d\times d}$, and $B(y(t;\mu);\mu)=B(\mu)\in\R^{d\times m}$, under standard stabilizability assumptions, the optimal control is expressed in feedback form 

\begin{equation}
   u(y;\mu) = -R^{-1} B^{\top}(\mu) P(\mu)y\,,
\end{equation}
where $P(\mu)\in\R^{d\times d}$ is the unique positive definite solution of the Algebraic Riccati Equation (ARE)
\begin{align}
A^\top(\mu) P(\mu) + P(\mu) A(\mu)
  -P(\mu)B(\mu)R^{-1}B^\top(\mu)P(\mu) +Q = 0\,.
    \label{are}
\end{align}
In this case, the feedback operator $K=R^{-1} B^{\top}(\mu) P(\mu)$ does not depend on the state. Formally, the SDRE method extends this logic to semilinear systems by indexing with respect to the current state of the trajectory, that is,
\begin{equation}
   u(y;\mu) = -R^{-1} B^{\top}(y;\mu) P(y;\mu)y\,,
    \label{control_sdre_feed}
\end{equation}
where $P(y;\mu)$ is the solution of the ARE
\begin{align}
A^\top(y;\mu) P(y;\mu) + P(y;\mu) A(y;\mu)-P(y;\mu)B(y;\mu)R^{-1}B^\top(y;\mu)P(y;\mu) & = -Q
    \label{sdre}
\end{align}
where $A(y;\mu)$ and $B(y;\mu)$ are frozen at $y$.
The implementation of the SDRE control requires the sequential solution of eq. \eqref{sdre} as the state $y(t;\mu)$ evolves in time. The SDRE feedback loop satisfies necessary optimality conditions at a quadratic rate as the state is driven to zero (we refer to \cite{ccimen2008state} for more details and a complete statement of the result).

\subsection{Tensor Train surrogate model for the control}
\label{sec:TT}
Neither of the two previously presented control laws can be effectively implemented for real-time stabilization of large-scale dynamics. Both the numerical realization of the PMP optimality system and the computational burden of solving AREs along a trajectory in the SDRE method surpass the time scale required for real-time control. Hence, we propose to construct a surrogate feedback law from state-parameter samples in a supervised learning manner. As a surrogate ansatz we use the Tensor Train (TT) decomposition, which can then be evaluated in real time along a controlled trajectory.

A scalar function $\tilde u(x): \mathbb{R}^d \rightarrow \mathbb{R}$ with $x=(x_1,\ldots,x_d) \in \mathbb{R}^d$ is said to be represented in the Functional Tensor Train (FTT) decomposition \cite{Marzouk-stt-2016,Gorodetsky-ctt-2019} if it can be written as
\begin{equation}\label{eq:ftt}
\tilde u(x) = \sum_{\alpha_0,\ldots,\alpha_{d}=1}^{r_0,\ldots,r_{d}} u^{(1)}_{\alpha_0,\alpha_1}(x_1) u^{(2)}_{\alpha_1,\alpha_2}(x_2) \cdots u^{(d)}_{\alpha_{d-1},\alpha_d}(x_d),
\end{equation}
with some \emph{cores} $u^{(k)}(x_k): \mathbb{R} \rightarrow \mathbb{R}^{r_{k-1} \times r_k}$, $k=1,\ldots,d$, and \emph{ranks} $r_0,\ldots,r_d$.
Without loss of generality, we can let $r_0=r_d=1$, but the intermediate ranks can be larger than $1$ to account for the structure of the function.
In practice, the exact decomposition \eqref{eq:ftt} may be not possible, and we seek to approximate a given control function $u(x)$ by $\tilde u(x)$ in the form \eqref{eq:ftt}, minimising the error $\|u - \tilde u\|$ in some (usually Euclidean) norm.
We aim at scenarios where $r_k$ stay bounded, or scale mildly.
Quadratic value functions \cite{DKK21} or weakly correlated Gaussian functions \cite{rdgs-tt-gauss-2022}, for example, admit FTT approximations with ranks at most polynomial in $d$ and poly-logarithmic in the approximation error $\|u - \tilde u\|$.
This allows one to avoid the curse of dimensionality by replacing $u(x)$ by its approximation \eqref{eq:ftt}.

The FTT \eqref{eq:ftt} has originally emerged as the TT decomposition \cite{osel-tt-2011} of tensors (multiindex arrays).
Indeed, the two come together.
For practical computations with \eqref{eq:ftt} one needs to discretise each core using some basis functions $\phi^{(k)}_1(x_k),\ldots,\phi^{(k)}_{n_k}(x_k)$ in the $k$-th variable.
One can now write the $k$-th core using a 3-dimensional tensor $\mathcal{U}^{(k)}\in\mathbb{R}^{r_{k-1} \times n_k \times r_k}$ of expansion coefficients,
\begin{equation}\label{eq:fttcore}
u^{(k)}_{\alpha_{k-1},\alpha_k}(x_k) = \sum_{i=1}^{n_k} \mathcal{U}^{(k)}(\alpha_{k-1},i,\alpha_k) \phi_i^{(k)}(x_k).
\end{equation}
The multivariate function $\tilde u(x)$ becomes discretised in the Cartesian basis, and defined by a $d$-dimensional tensor of expansion coefficients,
\begin{align}
\tilde u(x) & = \sum_{i_1,\ldots,i_d=1}^{n_1,\ldots,n_d} \mathcal{U}(i_1,\ldots,i_d) \phi^{(1)}_{i_1}(x_1) \cdots \phi^{(d)}_{i_d}(x_d), \nonumber \\
\mathcal{U}(i_1,\ldots,i_d) & = \sum_{\alpha_0,\ldots,\alpha_{d}=1}^{r_0,\ldots,r_{d}} \mathcal{U}^{(1)}(\alpha_0,i_1,\alpha_1) \cdots \mathcal{U}^{(d)}(\alpha_{d-1},i_d,\alpha_d).
\label{eq:discr-tt}
\end{align}
The tensor decomposition \eqref{eq:discr-tt} is the original TT format \cite{osel-tt-2011}.
Note that to interpolate $\tilde u(x)$ on any given $x$ requires first $d$ univariate interpolations of the cores as shown in \eqref{eq:fttcore},
followed by the multiplication of matrices $u^{(1)}(x_1) \cdots u^{(d)}(x_d)$.
This can be implemented in $2\sum_{k=1}^d n_k r_{k-1} r_k = \mathcal{O}(dnr^2)$ operations, where we define upper bounds $n:=\max_k n_k$ and $r:=\max_k r_k$.

To compute a TT approximation, instead of minimising the error $\|u-\tilde u\|$ directly, one can use faster algorithms based on the so-called \emph{cross interpolation} \cite{ot-ttcross-2010,ds-parcross-2020,dolgov2022data}.
These methods consider sampling sets of the Cartesian form $X_{<k} \oplus X_k \oplus X_{>k}$, for each $k=1,\ldots,d,$ where $X_{<k} = \{(x_1,\ldots,x_{k-1})\}$ of cardinality $r_k$, $X_k = \{x_k\}$ of cardinality $n_k$, and $X_{>k} = \{(x_{k+1},\ldots,x_d)\}$ of cardinality $r_k$;
$X_{<k}$ and $X_{>k}$ are taken empty when $k=1$ and $k=d$, respectively.
Note that the number of samples in $X_{<k} \oplus X_k \oplus X_{>k}$ is $r_{k-1} n_k r_k$, equal to the number of unknowns in $\mathcal{U}^{(k)}$.
Therefore, one can resolve the linear interpolation equations 
$$
u(x) = \tilde u(x)\,, \quad \forall x \in X_{<k} \oplus X_k \oplus X_{>k}
$$
on the elements of $\mathcal{U}^{(k)}$ exactly, whenever $u^{(1)}(x_1)\cdots u^{(k-1)}(x_{k-1})$ for $(x_1,\ldots,x_{k-1}) \in X_{<k}$, $\phi^{(k)}(x_k)$ for $x_k \in X_k$ and $u^{(k+1)}(x_{k+1}) \cdots u^{(d)}(x_d)$ for $(x_{k+1},\ldots,x_d) \in X_{>k}$ are linearly independent.
Moreover, pivoting can be applied to thus computed $\mathcal{U}^{(k)}$ to identify next sampling sets $X_{<k+1}$ and $X_{>k-1}$ as subsets of $X_{<k} \oplus X_k$ and $X_k \oplus X_{>k}$, respectively, to make $u^{(1)}(x_1)\cdots u^{(k)}(x_{k})$ and $u^{(k)}(x_{k}) \cdots u^{(d)}(x_d)$ in the next step better conditioned.
Iterating over $k=1,\ldots,d$, the TT-Cross updates all TT cores until convergence.
Vector functions can be approximated component by component, since the number of components in the control problems we consider is small.
For a comprehensive textbook on tensor methods see e.g. \cite{hackbusch-2012}.

\section{A statistical POD approach for optimal control problems}

In this section we describe our main contribution, outlined in Figure~\ref{fig:workflow}. Our ultimate goal is the fast synthesis of the feedback control that is resilient to stochastic perturbations of the system.
We split the entire procedure into two stages.
In the first (\emph{offline}) stage, we construct a reduced state basis that minimizes the average squared projection error for an ensemble of controlled trajectories of random realizations of the dynamical system and/or initial state. Due to this averaging of random samples of the error, we call the span of thus obtained basis the Statistical POD subspace.
It accommodates states that are well controllable on average, and hence its dimension can be much smaller than the original problem dimension.
In the second (\emph{online}) stage, the system is projected onto the Statistical POD subspace, and an optimal control problem is solved for the reduced system in the feedback form.
Due to its smaller dimension, the reduced optimal control problem is faster to solve numerically, and can enable the real-time control synthesis.
The computations can be accelerated even further if the control function of the reduced problem is approximated in the TT format.
The latter step can be moved into the offline stage, making the online stage a mere TT interpolation.

\vspace{5mm}

{\bf Offline Stage}

The offline phase for the statistical POD method constitutes of the collection of snapshots and the construction of the reduced basis and, hence, the reduced dynamics. In this phase we rely on the application of the PMP system \eqref{pmp}-\eqref{pmp2} to obtain the optimal trajectory for each given sample of the parameter. We fix $N$ different realizations $\underline{\mu} = (\mu_1,\ldots,\mu_N) \in \R^{M \times N}$. For each realization $\mu_i \in \underline{\mu}$, we solve the system \eqref{pmp}-\eqref{pmp2} and we collect the corresponding optimal trajectory $(y^*(t_1;\mu_i), \ldots, y^*(t_{n_t};\mu_i))$ at some $n_t$ time instances $\{t_i\}_{i=1}^{n_t}$. The final snapshot matrix $Y_{\underline{\mu}}$ will contain all the optimal trajectories at the given time instances $t_1,\ldots,t_{n_t}$ and realizations $\mu_1,\ldots,\mu_N$:
\begin{equation}\label{eq:all_snapshots}
Y_{\underline{\mu}}=[y^*(t_1;\mu_1), \ldots,y^*(t_{n_t};\mu_1),y^*(t_{1};\mu_2), \ldots, y^*(t_{n_t};\mu_N)] \in \mathbb{R}^{d \times n_t N},
\end{equation}
where each $y^*(t_i;\mu_j) \in \R^d$ is a column vector.

At this point we perform a Singular Value Decomposition (SVD) of the snapshot matrix $Y_{\underline{\mu}} = U_{\underline{\mu}} \Sigma_{\underline{\mu}} V_{\underline{\mu}}^\top$ and we denote as $U_{\underline{\mu}}^\ell$ the first $\ell$ columns of the orthogonal matrix $U_{\underline{\mu}}$. The number of basis $\ell$ can be selected by making the truncation error
\begin{equation}
\mathcal{E_{\underline{\mu}}}(\ell) =\frac{ \sum_{i=1}^{\ell }\sigma^2_{i,{\underline{\mu}}}}{\sum_{i=1}^{\min\{N n_t, d\} } \sigma^2_{i,{\underline{\mu}}}}
\label{err_ind}
\end{equation}
below a desired threshold.
Here, $\{\sigma_{i,{\underline{\mu}}}\}_i$ are the singular values of the matrix $Y_{\underline{\mu}}$.
The quantity \eqref{err_ind} is related to the projection error arising in the truncation of the singular value decomposition to the first $\ell$ basis. To shorten the notation, in what follows, we will denote the reduced basis matrix as $U^\ell := U_{\underline{\mu}}^{\ell}$ .

The full order dynamics \eqref{eq} is then reduced via projection onto the subspace spanned by $U^\ell$ to the following reduced dynamics

\begin{equation}\label{eq_red}
\left\{ \begin{array}{l}
\frac{d}{dt}y^\ell(t;\mu)=(U^\ell)^\top f(U^\ell y^\ell(t;\mu),u(t;\mu),\mu), \;\; t\in(0,+\infty),\mu \in \underline{\mu},\\
y^\ell(0;\mu)=(U^\ell)^\top 
 x  \in\mathbb{R}^\ell.
\end{array} \right.
\end{equation}

To ease the notation, we define 
$$
f^\ell (y^\ell(t;\mu),u(t;\mu),\mu) = (U^\ell)^\top f(U^\ell y^\ell(t;\mu),u(t;\mu),\mu).
$$
The procedure for the offline stage is presented in Algorithm \ref{alg_1}.

\begin{algorithm}[H]
\caption{Offline Stage for the Statistical POD method}
\begin{algorithmic}[1]
\State Fix the final time $T$, a time discretization $\{t_i\}_{i=1}^{n_t}$, $N$ realizations $\{\mu_i\}_{i=1}^N$, truncation threshold $\varepsilon \ge 0$ and $Y_{\underline{\mu}}=[\,]$
\For{$i=1,\ldots,N$}\label{for_start}
\State{Solve the PMP system \eqref{pmp}-\eqref{pmp2}}
\State{Compute the optimal trajectory $Y_{\mu_i} = [y^*(t_1;\mu_i), \ldots,y^*(t_{n_t};\mu_i)]$}
\State{$Y_{\underline{\mu}}=[Y_{\underline{\mu}} \; Y_{\mu_i}]$}
\EndFor\label{for_end}
\State{Perform the SVD $Y_{\underline{\mu}} = U_{\underline{\mu}} \Sigma_{\underline{\mu}} V_{\underline{\mu}}^\top$}
\State{Select $\min\ell:~\mathcal{E_{\underline{\mu}}}(\ell)<\varepsilon^2$ \eqref{err_ind} and take the first $\ell$ columns of $U_{\underline{\mu}}$ into $U^\ell.$}
\State{Construct the reduced dynamics \eqref{eq_red}}
\end{algorithmic}
\label{alg_1}
\end{algorithm}

\begin{remark}
In general, $f(U^\ell y^\ell(t;\mu),u(t;\mu),\mu)$ is still computationally expensive since it requires the evaluation of the nonlinearity on the lifted variable $U^\ell y^\ell(t;\mu)\in \R^d$. For these cases one may apply Empirical Interpolation Method (EIM, \cite{barrault2004empirical}) and Discrete Empirical Interpolation Method (DEIM, \cite{chaturantabut2010nonlinear}) to solve this issue.
\end{remark}

\begin{remark}
In some cases, such as problems arising from semidiscretization of nonlinear PDEs, the dynamical system may present a quadratic-bilinear form:
\begin{equation}
    f(y(t;\mu),u(t;\mu),\mu) = A y + {\bf T} (y \otimes y )  + \sum_{i=1}^m  \theta_k(\mu) N_k y_k u_k + \theta_0(\mu) B u,
\end{equation}
where $A, N_k \in \R^{d \times d}$, ${\bf T} \in \R^{d \times d^2}$, $B \in \R^{d \times m}$ and $\theta_i: \mu \rightarrow \R$. In this case, we can assemble the reduced matrices
$$
A^\ell = (U^\ell)^\top A U^\ell,  \;N_k^\ell = (U^\ell)^\top N_k U^\ell,  \; {\bf T}^\ell = (U^\ell)^\top {\bf T} (U^\ell \otimes U^\ell), \, B^\ell = (U^\ell)^\top B
$$
and define the following reduced dynamics
\begin{equation}
\frac{d}{dt} y^\ell(t;\mu)= A^\ell y^\ell + {\bf T}^\ell (y^\ell \otimes y^\ell )  + \sum_{i=1}^m  \theta_k(\mu) N^\ell_k y^\ell_k u_k + \theta_0(\mu) B^\ell u,
\label{red_quad_bil}
\end{equation}
avoiding the use of empirical interpolation techniques.
\end{remark}

\vspace{5mm}

{\bf Online Stage}

Given the reduced dynamics \eqref{eq_red}, the reduced PMP for a given $\mu$ reads

\begin{align}\label{pmp_red}
\frac{d}{dt}y^\ell(t,\mu)&=f^\ell (y^\ell(t;\mu),u(t;\mu),\mu), \\
y^\ell(0;\mu)&=y^\ell_0(\mu),\\
-\frac{d}{dt} p^\ell(t;\mu)&=  \nabla_{y^\ell} f^\ell(y^\ell(t;\mu),u(t;\mu),\mu)^\top p^\ell(t;\mu)+\nabla_{y^\ell} L(U^\ell y^\ell(t;\mu),u(t;\mu)), \\
p^\ell(T;\mu)&=0 , \\
u(t)&=\argmin_{w \in U_{ad}}  \{L(U^\ell y^\ell(t;\mu),w) +f^\ell(y^\ell(t;\mu),w,\mu)^{\top} p^\ell(t;\mu)  \}\,.
\label{pmp_red2}
\end{align}

The reduced-order state variable is $\ell$-dimensional and it can be exploited to construct a faster approximation of the optimal control. The reduced-order dynamics enables the synthesis of an optimal feedback law, now set in a lower dimensional domain, helping in the mitigation of the curse of dimensionality.

Let us consider the feedback map $u^\ell: \mathbb{R}^ \ell \rightarrow U_{ad}$ obtained solving the reduced PMP \eqref{pmp_red}-\eqref{pmp_red2}. This controller can be used in the full order model dynamics considering first the projection of the state and then the evaluation of the feedback map:
\begin{equation}
\dot{y}(t;\mu)=f(y(t;\mu),u^\ell((U^\ell)^\top y(t;\mu )),\mu),
\label{ode_feedback}
\end{equation}
obtaining a control problem where the computation of the optimal feedback is independent from the original dimension of the dynamical system.

\begin{remark}
    The presented methodology is based on the construction of a snapshot set of optimal trajectories obtained via the resolution of the PMP system \eqref{pmp}-\eqref{pmp2}. The sampling of the controlled dataset may be derived also upon the SDRE framework introduced in Section \ref{sec:sdre}, solving the sequential AREs \eqref{sdre} along the optimal trajectory. Since the number of variables in the dynamical system may be arbitrary large, the SDRE approach requires efficient solvers for the resolution of high-dimensional Riccati equations. We refer to \cite{BBKS_2020} for a comprehensive discussion of such techniques.
\end{remark}

\begin{remark}
The singular values for the Statistical POD method may present a slow decay due to the nature of the problem and the size of the parameter domain. In this context a good approximation requires a large number of POD basis and the online stage will still be affected by the dimensionality problem. To this end, an efficient surrogate model for the feedback control in the reduced space accelerates the procedure, obtaining a fast and reliable method for the computation of the optimal control solution.  
We consider the representation of the feedback control $u^\ell(x^\ell)$ in the Functional Tensor Train format \eqref{eq:ftt}, obtained via a TT Cross procedure (see Section \ref{sec:TT}). The optimal reduced trajectories obtained by solving the system \eqref{pmp_red}-\eqref{pmp_red2} are employed as sampling points for the Cross approximation. 
\end{remark}

\begin{remark}
In contrast to existing POD techniques for the approximation of optimal control problems, the proposed Statistical POD method introduces stochastic terms in the dynamical system to explore the manifold of controlled solutions for the parameterized problem. This enforces the robustness of the corresponding reduced basis to perturbations, an essential feature for feedback control.
In particular, as shown in Section \ref{sec:err}, the empirical risk of the Statistical POD procedure converges to the expected risk. Thus, increasing the amount of training data for constructing the reduced-order system, we obtain a model that is more representative of the true system dynamics, including unseen scenarios.
\end{remark}
\subsection{Efficient computation of the snapshots}
\label{eff_snap}

The main building block in the offline phase is represented by the computation of different optimal trajectories via the PMP system \eqref{pmp}-\eqref{pmp2}. The resolution of this system relies strongly on the choice of the initial guess for the optimal control $u^*(\cdot)$. In this section we introduce a procedure which accelerates the high-dimensional computation exploiting the information about the optimum for the reduced problem. Let us consider a prefixed initial guess $u^0(\cdot)$. In the offline stage presented in Algorithm \ref{alg_1}, the different optimal trajectory realizations are stored in the snapshot matrix $Y_{\underline{\mu}}$ and the basis are constructed only at the end of the loop. An SVD decomposition can be added between Lines~\ref{for_start}-\ref{for_end} of Alg.~\ref{alg_1} to construct a reduced basis containing partial information on the statistical behaviour of the optimal trajectories and build a reduced dynamics. The constructed reduced system is employed for a fast computation of the approximated optimal control $u_{red}(\cdot)$, which can be used as initial guess for the high-dimensional computation. In particular, at the $i$-th step we consider as initial guess the control which achieves a lower cost functional, $i.e.$ 
\begin{equation}
u^i(\cdot) \in \argmin_{u(\cdot) \in \{u^0(\cdot),u_{red}^i(\cdot) \}} J(u) .
\label{guess}
\end{equation}

The optimized procedure is sketched in Algorithm \ref{alg_2}.

\begin{algorithm}[H]
\caption{Optimized Offline Stage for the Statistical POD method}
\begin{algorithmic}[1]
\State Fix the final time $T$, a time discretization $\{t_i\}_{i=1}^{n_t}$, $N$ realizations $\{\mu_i\}_{i=1}^N$, truncation threshold $\varepsilon\ge 0$, an initial control guess $u^0$ and $Y_{\underline{\mu}}=[\,]$
\For{$i=1,\ldots,N$}
\If{$i==1$}
\State{Solve the PMP system \eqref{pmp}-\eqref{pmp2} with initial guess $u^0$}
\Else{}
\State{Solve the reduced PMP system \eqref{pmp_red}-\eqref{pmp_red2} and compute $u_{red}^i$}
\State{Solve the PMP system \eqref{pmp}-\eqref{pmp2} with initial guess given by \eqref{guess}}
\EndIf
\State{Compute the optimal trajectory $Y_{\mu_i} = [y^*(t_1;\mu_i), \ldots,y^*(t_{n_t};\mu_i)]$}
\State{$Y_{\underline{\mu}}=[Y_{\underline{\mu}} \; Y_{\mu_i}]$}
\State{Perform the SVD $Y_{\underline{\mu}} = U_{\underline{\mu}} \Sigma_{\underline{\mu}} V_{\underline{\mu}}^\top$}
\State{Select $\min\ell:~\mathcal{E_{\underline{\mu}}}(\ell)<\varepsilon^2$ and take the first $\ell$ columns of $U_{\underline{\mu}}$ into $U^\ell.$}
\State{Construct the reduced dynamics \eqref{eq_red}}
\EndFor

\end{algorithmic}
\label{alg_2}
\end{algorithm}

\section{Expected risk of statistical POD}
\label{sec:err}

We analyse the proposed technique providing a rigorous error estimate on the expected risk. For this reason we first recall an eigenvalue perturbation theory, the Bauer-Fike theorem, which will be employed for the proof of Proposition \ref{prop1}.

\begin{theorem}[Bauer-Fike]
Let $A$ be an $d \times d$ diagonalizable such that $A = V \Lambda V^{-1}$ and let $E$ be an arbitrary $d \times d$ matrix. Then for every $\mu \in \sigma(A+E)$ there exists $\lambda \in \sigma(A)$ such that
\begin{equation*}
    |\lambda - \mu| \le \kappa_p(V) \Vert E \Vert_p
    \label{BF}
\end{equation*}
    where $\kappa_p(V)$ is the condition number in $p$-norm of the matrix $V$.
    \label{BFtheorem}
\end{theorem}

We are ready to formulate the error estimate on the expected projection error characterizing the statistical POD approach.
\begin{prop}
Consider $t$ as a uniformly distributed random variable on $[0,T]$, and consider $N$ iid realizations of the optimal trajectories $y^{(i)}:=y(t_i;\mu_i)$, $i=1,\ldots,N$.
Assume that $\sqrt{\mathrm{Var}[y_jy_k]} \le C<\infty$ for any $j,k=1,\ldots,d$, 
where $y_j$ is the $j$th component of the random vector $y(t,\mu)$.
Let
$$
G = \frac{1}{N}\sum_{i=1}^{N} y(t_i;\mu_i) y(t_i;\mu_i)^\top \quad \mbox{and} \quad G^* = \mathbb{E}[yy^\top],
$$
and let $U^{\ell} \in \mathbb{R}^{d \times \ell}$ be an orthonormal matrix of $\ell$ leading eigenvectors of $G$.
Then 
\begin{equation}\label{eq:er}
\mathbb{E}[\|y(t;\mu) - U^{\ell} (U^{\ell})^\top y(t;\mu)\|_2^2] \le \sum_{j=\ell+1}^{d} \lambda^*_j + 2d \ell \frac{C}{\sqrt{N}}
\le \sum_{j=\ell+1}^{d} \lambda_j + (d+\ell)d \frac{C}{\sqrt{N}},
\end{equation}
where $\lambda^*_j$ are eigenvalues of $G^*$ sorted descending,
and $\lambda_j$ are eigenvalues of $G$ sorted descending.
\label{prop1}
\end{prop}

\begin{proof}

Let $G^* = U_* \Lambda_* U_*^\top$ and $G = U \Lambda U^\top$ be the eigenvalue decompositions of $G^*$ and $G$, respectively,
where eigenvalues in $\Lambda$ and $\Lambda_*$ are sorted descending.
Note that the matrix $U^\ell$ contains the first $\ell$ columns of $U$. Moreover, let us introduce $P^\ell = U^\ell (U^\ell)^\top$ to shorten the notation.

Since $y(t;\mu)$ in \eqref{eq:er} is independent of precomputed trajectories $y^{(i)}$, and hence of $G$ and $U^{\ell}$, we can factorise the total expectation into those over $y$ and $G$,
$$
\mathbb{E}[\|y - U^{\ell} (U^{\ell})^\top y\|_2^2] = \mathbb{E}_G[\mathbb{E}_y[\|y - U^{\ell} (U^{\ell})^\top y\|_2^2]],
$$
and hence express first
$$
\mathbb{E}_y[\|y - P^\ell y\|_2^2] = \mathbb{E}_y[y^\top y - y^\top P^\ell y] = \mathrm{tr}(\mathbb{E}_y[yy^\top -  P^\ell y y^\top])
$$
due to orthogonality and the cyclic permutation under the trace of a matrix.
Now,
$$
\mathrm{tr}(\mathbb{E}_y[yy^\top]) =  \mathrm{tr}(\Lambda_*) =   \sum_{j=1}^{d} \lambda_j^*,
$$
while
\begin{align*}
\mathrm{tr}(\mathbb{E}_y[P^\ell  yy^\top]) & = \mathrm{tr}(P^\ell G^*)
 = \mathrm{tr}(P^\ell G) + \mathrm{tr}(P^\ell (G^* - G))  \\
& = \mathrm{tr}((U^\ell)^\top U \Lambda U^\top U^\ell)) + \mathrm{tr}(P^\ell (G^* - G)).
\end{align*}
Due to orthogonality, the first trace is just the trace of the leading $\ell \times \ell$ submatrix of $\Lambda$, which is $\sum_{j=1}^{\ell} \lambda_j$.
Now taking also $\mathbb{E}_G$, we obtain
\begin{align*}
\mathbb{E}[\|y - P^\ell y\|_2^2] & = \sum_{j=1}^{d} \lambda_j^* - \mathbb{E}_G[\sum_{j=1}^{\ell} \lambda_j] - \mathbb{E}_G[\mathrm{tr}((U^\ell)^\top (G^*-G) U^\ell)] \\
& = \sum_{j=\ell+1}^{d} \lambda_j^* + \mathbb{E}_G[\sum_{j=1}^{\ell} (\lambda_j^* - \lambda_j)] - \mathbb{E}_G[\mathrm{tr}((U^\ell)^\top (G^*-G) U^\ell)] \\
& \le  \sum_{j=\ell+1}^{d} \lambda_j^* + \ell \cdot \mathbb{E}_G[\|G^* - G\|_2] + \ell \cdot \mathbb{E}_G[\|G^* - G\|_2],
\end{align*}
due to the Bauer-Fike Theorem \ref{BFtheorem} (second term), and 
$$
|\mathrm{tr}(Q^\top A Q)| = |\sum_{j=1}^{\ell} q_j^\top A q_j| \le \ell \max_{j \in \{1,\ldots,\ell\}} |q_j^\top A q_j| \le \ell \|A\|_2
$$
for any $A \in \mathbb{R}^{d \times d}$ and orthonormal $Q \in \mathbb{R}^{d \times \ell}$
for the third term.

By the Jensen's inequality and norm equivalence,
$$
(\mathbb{E}_G[\|G^* - G\|_2])^2 \le \mathbb{E}_G[\|G^* - G\|_2^2] \le \mathbb{E}_G[\|G^* - G\|_F^2] = \sum_{j,k=1}^d \mathbb{E}_G[(G^*_{j,k} - G_{j,k})^2].
$$
Since $\mathbb{E}_G[G] = G^*$, $\mathbb{E}_G[(G^*_{j,k} - G_{j,k})^2] = \mathrm{Var}[G_{j,k}].$
In turn, $G_{j,k}$ is a sum of iid random variables $y^{(i)}_j y^{(i)}_k/N$.
For those,
$$
\mathrm{Var}[G_{j,k}] = N \cdot \mathrm{Var}[y^{(i)}_j y^{(i)}_k/N] = \frac{1}{N} \mathrm{Var}[y_j y_k] \le \frac{C^2}{N},
$$
where in the penultimate step we used that $y^{(i)}$ are samples from the same distribution of $y$.
Summing over $j,k$ and taking the square root gives the first claim of the proposition.
The second claim comes from rewriting
$$
\sum_{j=\ell+1}^{d} \lambda^*_j = \sum_{j=\ell+1}^{d} \lambda_j + \sum_{j=\ell+1}^{d} (\lambda^*_j - \lambda_j),
$$
and using again Theorem \ref{BFtheorem} to get $|\lambda^*_j - \lambda_j| \le \|G^* - G\|_2$.
\end{proof}

Note that the expected risk consists of the truncated eigenvalues (bias), as for the classical POD technique, and a variance term decaying as $1/\sqrt{N},$ as usual for Monte-Carlo methods.

\vspace{5mm}

\section{Neumann boundary control for 2D Burgers' equation}

To better illustrate the different blocks of our approach, in this section we study a stabilization problem for the solution of the viscous Burgers' equation in a bidimensional domain, where the control appears as a Neumann boundary condition. 

\subsection{Problem formulation}
More precisely, we consider the following state equation 
\begin{equation}
\begin{cases}
\partial_t y -\nu \Delta y + y \cdot \nabla y =\alpha y & (\xi,t,\mu) \in \Xi \times [0,T] \times \R^M, \\
\nu \partial_n y = u(t) & (\xi,t,\mu) \in \partial \Xi_1  \times [0,T] \times \R^M, \\
\nu \partial_n y = 0 & (\xi,t,\mu) \in \partial \Xi_2 \times [0,T] \times \R^M, \\
y(\xi,0;\mu) = \tilde y_0(\xi;\mu) &  (\xi,\mu) \in \Xi \times \R^M, \\
\end{cases}
\label{burgers}
\end{equation}
where $\nu >0$ is the viscosity constant, $\alpha \ge 0$ regulates the unstable term, $\xi=(\xi_1,\xi_2) \in \Xi = [0,1]^2$ is the spatial variable, $\partial \Xi_1 = \{0\} \times [0,1]$  and $ \partial \Xi_2 = \partial \Xi \setminus \partial \Xi_1$, see Figure \ref{fig:square}.
Our aim is to drive the dynamical system to the equilibrium $\overline{y} \equiv 0$ and the corresponding cost functional we want to minimize reads:
\begin{equation}
J_T(u; \mu) = \int_0^{T} \int_\Xi |y(\xi,t;\mu)|^2 \,d\xi \, dt + \int_0^{T} |u(t)|^2 \; dt .
\label{cost_burgers}
\end{equation}

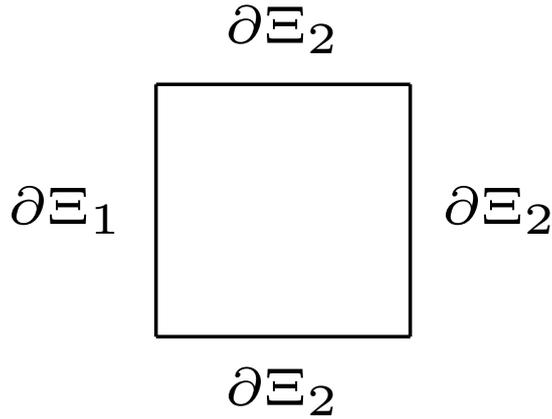
\begin{figure}[t!]
\centering
\resizebox{0.5\linewidth}{!}{
\begin{tikzpicture}{\small
 \draw (0,0) -- (0,1) node[left,midway] {\tiny{$\partial \Xi_1$}}; 
\draw (0,0) -- (1,0) node[below,midway] {\tiny{$\partial \Xi_2$}};
\draw (0,1) -- (1,1) node[above,midway] {\tiny{$\partial \Xi_2$}};
\draw (1,0) -- (1,1) node[right,midway] {\tiny{$\partial \Xi_2$}};
}
\end{tikzpicture}
}
\caption{Domain for the Burgers' equation \eqref{burgers}.}
\label{fig:square}
\end{figure}

For the application of the statistical POD technique, we assume that the initial condition is a realisation of the following random field,
\begin{equation}
\tilde{y}_0(\xi; \mu) =  y_0(\xi) +  \sum_{i=1}^{M_1} \sum_{j=1}^{M_2} (i+j)^{-\gamma} \mu_{i+(j-1)M_1} \cos(i \pi \xi_1) \cos(j \pi \xi_2),
\label{ic_stat}
\end{equation}
where $\mu_{i+(j-1)M_1} \sim \mathcal{N}(0, \sigma^2)$ are normally distributed random variables with zero mean and variance $\sigma^2$,
$M_1$ and $M_2$ are the maximal frequencies in the first and second spatial variables,
and $y_0(\xi)$ is the mean initial state.
We take $\sigma=\sigma_0=0.05$ and $y_0 \equiv 0.05$ by default.
The parameter $\gamma$ refers to the decay of the Fourier coefficients and it is related to the regularity we want to assume. In Figure \ref{icgamma} we can observe two initial conditions fixing $\gamma=3$ on the left and $\gamma=4$ on the right. Higher values for $\gamma$ correspond to smoother initial conditions since the Fourier coefficients decay more rapidly.

\begin{figure}[htbp]	
\centering
	\includegraphics[width=0.49\textwidth]{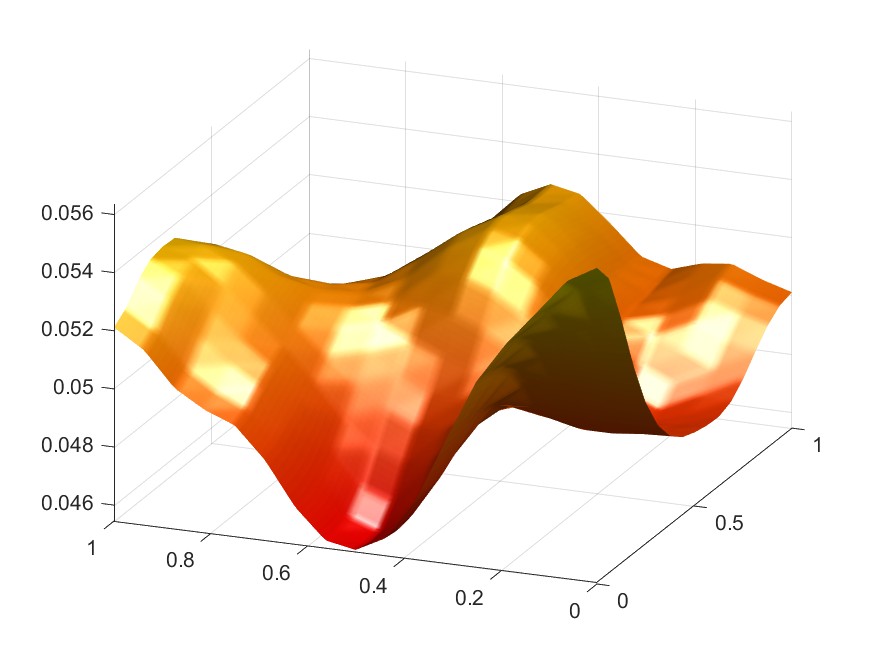}
	\includegraphics[width=0.49\textwidth]{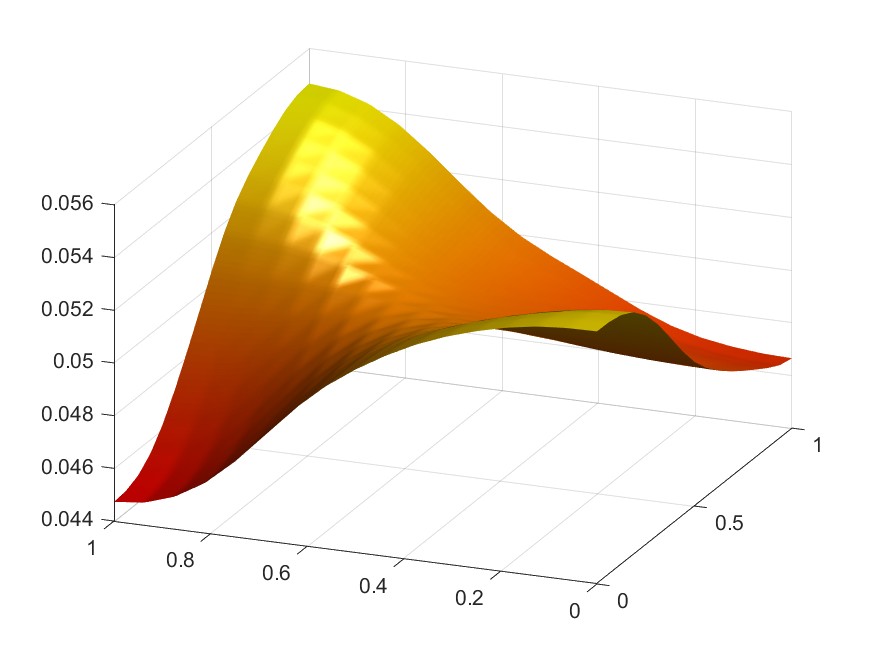}
\caption{Observed initial conditions for $\gamma = 3$ (left) and $\gamma=4$ (right).}	
\label{icgamma}
\end{figure}

First, we discretize the Burgers' equation \eqref{burgers} via $P^1$ finite elements $\{\phi_i\}_{i=1}^d$ in space centered at uniform grid points $\{(\xi^i_1,\xi^i_2)\}_{i=1}^d$, obtaining the following semidiscrete form
\begin{equation}
E \dot{y}(t) = (C+\alpha E) y(t) + F(y(t) \otimes y(t))+ B u(t),
\label{ode_burgers}
\end{equation}
where $E\in\R^{d \times d}$ is the mass matrix, $C\in\R^{d \times d}$ is the stiffness matrix, $F \in \R^{d \times d^2}$ arises from the nonlinear convective term and
$$
(B)_{i} = \begin{cases} 1 & (\xi^i_1,\xi^i_2)  \in \partial \Xi_1, \\
     0 & \mbox{otherwise,} \\
\end{cases} \quad i=1,\ldots,d.
$$
We apply an implicit Euler scheme for the time discretization. In particular, we consider $d=289$ finite elements and $n_t = 100$ time steps, fixing the final time $T=50$. The computation of the optimal trajectories is based on the Pontryagin's Maximum Principle system \eqref{pmp}-\eqref{pmp2}, solved via a gradient-descent method for each sample of $\mu$. More precisely, we consider the state-adjoint system and the gradient of the cost functional following \cite{heinkenschloss2008numerical}.

Once the reduced basis are computed, following the construction of the reduced dynamics \eqref{red_quad_bil}, the reduced system reads
\begin{align}\label{burg_red}
\frac{d}{dt}{y^\ell}(t;\mu) & = A^\ell y^\ell(t;\mu) + F^\ell (y^\ell \otimes y^\ell) + B^\ell u(t), \\
y^\ell(0;\mu) & = x^\ell = (U^\ell)^\top x, \nonumber
\end{align}
 which can be rewritten in the following semilinear form
\begin{equation}
\dot{y^\ell}(t;\mu) = \mathcal{A}^\ell(y^\ell) y^\ell(t;\mu)  + B^\ell u(t),
\label{burg_red_semi}
\end{equation}
where 
$$
(\mathcal{A}^\ell(y))(i,j) =  A^\ell(i,j) + \sum_{k=1}^{\ell} F^\ell(i,(j-1)\ell+k) \, y_k ,\quad i,j \in \{1,\ldots, \ell\}.
$$
Note that the reduced system \eqref{burg_red} is free from $\mu$ (which appears only the full initial state $\tilde y_0$), hence \eqref{burg_red} can be written in the usual feedback form where the reduced initial state is a variable $x^\ell$.
For the remaining part of the section we fix default $\gamma=4$, $M_1=M_2=8$, $\nu=0.02$ and $T=50$.

\subsection{Error indicators.} 
Throughout the numerical experiments, we will measure the error of the reduced model using the following error indicators:

\begin{equation}\label{eq:err-model}
\mathcal{E}_{J} = |J(u^*)-J(u_{red}^*)|, \quad \mathcal{E}_y = \sum_{k=1}^{n_t} (t_{k}-t_{k-1}) \Vert y(u^*,t_k)-y( u_{red}^*,t_k)\Vert_2,
\end{equation}
where $J(u)$ is the total cost of the full model computed on the control $u$, $u^*$ is the optimal control of the full model, $u_{red}^*$ is the optimal control of the reduced model, while $y(u,t)$ is the full system state with control $u$ at time $t$.
In other terms, $\mathcal{E}_{J}$  and $\mathcal{E}_y$ test respectively the cost and the trajectory in the original model when the control signal is computed using the reduced model. The TT performance will be assessed using
\begin{equation}\label{eq:err-tt}
    \mathcal{E}_{TT} = |J(u_{red,TT}) - J(u_{red})|,
\end{equation} 
where $u_{red}$ is any feasible feedback law for the reduced system (not necessarily optimal), and $u_{red,TT}$ is its TT approximation.
For simplicity, in the numerical tests we use the SDRE controller.

\subsection{SPOD accuracy in the stable case ($\alpha = 0$)}

We begin considering the stable case $\alpha=0$, avoiding the exponential growth for the solution, while in the second subsection we analyse the behaviour for the unstable case.
We apply the statistical POD strategy for the reduction of \eqref{burgers}-\eqref{cost_burgers} with $\alpha = 0$ and we compare it with the POD methods based on the information of two specific dynamics: the uncontrolled solution and the optimal trajectory starting from the reference initial condition $y_0$. The training for the statistical POD is based on $N=40$ independent optimal trajectories each started from a sample of the initial condition \eqref{ic_stat}.

In the left panel of Figure \ref{err_burg_st} we show the decay of singular values of $Y_{\underline{\mu}}$ varying the number of samples $N \in \{1,20,40\}$. As expected, inclusion of more snapshots from different model regimes produces a slower decay.

However, singular vectors of $Y_{\underline{\mu}}$ with more snapshots produce a more accurate reduced model.
In the right panel of Figure \ref{decay_sv} we compare the model reduction errors \eqref{eq:err-model} for three different sets of snapshots used for POD: 
$Y_{\underline{\mu}}$ consisting of $N=40$ controlled trajectories started from random samples of the initial condition ("statistical POD"),
one controlled trajectory started from the mean initial condition ("1 POD controlled"), and
one uncontrolled trajectory started from the mean initial condition ("1 POD uncontrolled").
In all three cases $\ell=20$ dominant singular vectors of the snapshot matrix are selected as the basis for model reduction.
We see that the uncontrolled snapshots give a very inaccurate basis for reduction of the controlled system,
using the controlled snapshots gives a more accurate model, especially when multiple random samples of snapshots are used.

\begin{figure}[htbp]	
\centering
	\includegraphics[width=0.49\textwidth]{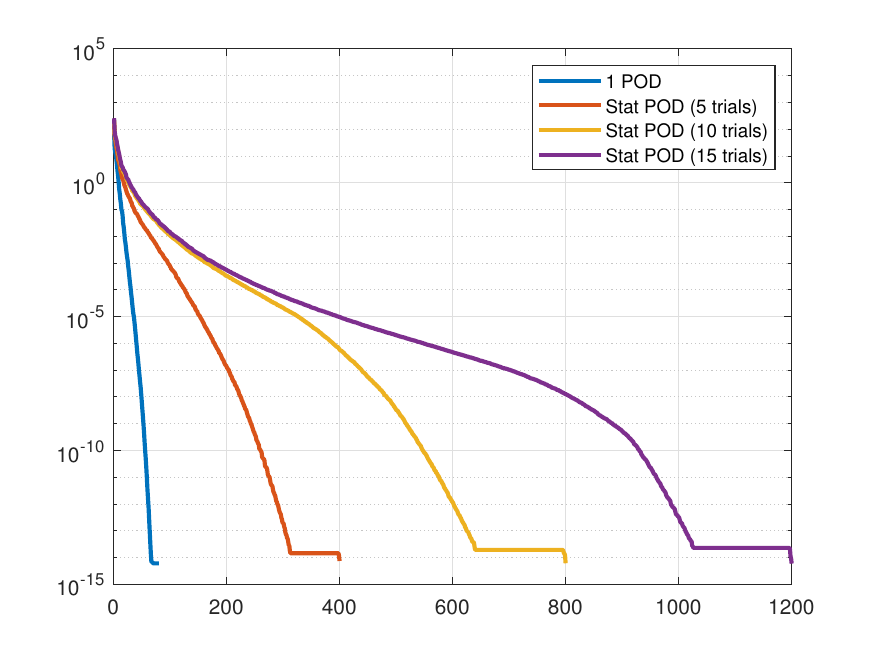}
\includegraphics[width=0.49\textwidth]{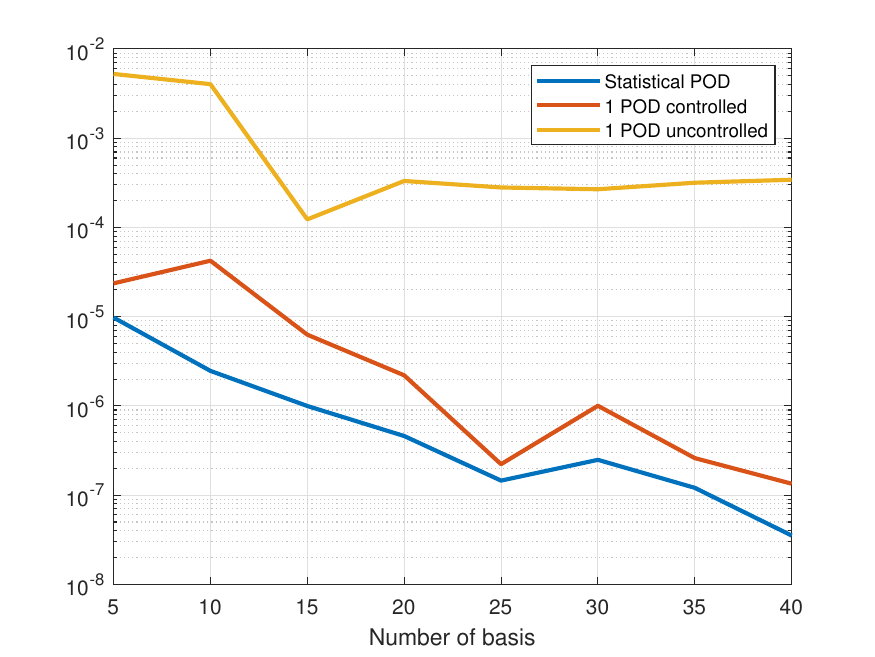}
\caption{Left: singular values of $Y_{\underline{\mu}}$ for different numbers $N$ of realisations of initial conditions. 
Right: mean error in the cost $\mathcal{E}_{J}$ depending on the reduced basis size $\ell$ for statistical POD built upon $N=40$ controlled trajectory realisations started from random iid samples from \eqref{ic_stat}, and the POD built upon 1 realisation of controlled and uncontrolled trajectory started from the mean $y_0$.
}
\label{decay_sv}
\end{figure}

 In Figure \ref{err_burg_st} we compare the cumulative CPU time for the computation of the offline phase with Algorithm~\ref{alg_1} and with its optimized version, Algorithm~\ref{alg_2}. The "not-optimized" algorithm has an almost constant increase in the CPU time and at the first iterations it performs slightly better since the statistical basis constructed with the first snapshots do not help in the construction of a good initial guess. On the other hand, the increase of the knowledge on the ensemble of reduced dynamics leads to a decrease of the CPU time at each step for the optimized version, demonstrating its beneficial support for the high-dimensional resolution of the PMP system.

\begin{figure}[htbp]	
\centering
	 \includegraphics[width=0.49\textwidth]{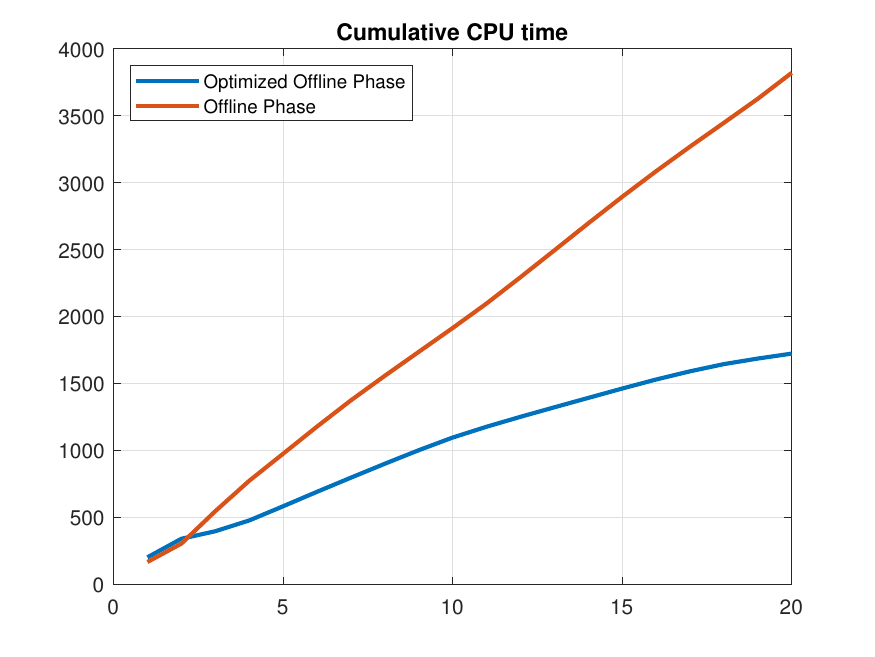}
\caption{Comparison of the CPU time for the offline phase and the optimized offline phase increasing the number of samples.}
\label{err_burg_st}
\end{figure}

In Table \ref{table_Burgers2} we compare the accuracy of the three techniques changing the parameters of the problem. In the first case we double the standard deviation for the random variables, such that $\mu_{i}\sim \mathcal{N}(0, 4\sigma_0^2)$, in the second case we consider $\gamma = 3$, obtaining a set of initial conditions less smooth. The dimension of the reduced basis in each case is $\ell=6$. The uncontrolled dynamics is not sufficient to build a good reduced model.

The POD technique based on controlled solutions obtains better results for both indicators, especially for the statistical POD. In particular, fixing $\gamma = 3$ the statistical POD reveals an improvement of one order of magnitude.

\begin{table}[hbht]
\centering
\begin{tabular}{c|cc|cc|cc}     
Error & \multicolumn{2}{c|}{Uncontrolled}  &  \multicolumn{2}{c|}{1 POD} & \multicolumn{2}{c}{SPOD} \\ 
     & $\sigma=2\sigma_0$ & $\sigma=\sigma_0$ & $\sigma=2\sigma_0$ & $\sigma=\sigma_0$ & $\sigma=2\sigma_0$ & $\sigma=\sigma_0$ \\
     & $\gamma=4$ & $\gamma=3$ & $\gamma=4$ & $\gamma=3$ & $\gamma=4$ & $\gamma=3$ \\\hline
$\mathcal{E}_y$ &  0.052 & 4.3e-3   &  1.11e-3 & 1.1e-3 &  3.17e-4 & 2.0e-4\\
$\mathcal{E}_{J}$ & 0.047 & 5.2e-3   & 1.65e-3 & 2.3e-4 &  1.58e-4 & 1.1e-4\\
 \end{tabular}
  \caption{Mean errors in the cost functional and optimal trajectory started from random iid samples from \eqref{ic_stat} with modified standard deviation or $\gamma$.
  The reduced basis size is $\ell=6$ for all methods.
  Statistical POD (SPOD) is computed from $N=20$ trajectories started from samples from the original \eqref{ic_stat} with $\sigma=\sigma_0$ and $\gamma=4$.
  }
 \label{table_Burgers2}
\end{table}

In Table \ref{diff_y0} we compare the 1 POD technique and the statistical POD with completely different initial data which cannot be represented in the form \eqref{ic_stat}. For these cases we fix the reduced basis size $\ell=30$ to obtain accurate approximations. We note that for all the test cases the statistical approach is more accurate than 1 POD, with a difference of almost one order of magnitude. Furthermore, we see that the accuracy is decreasing for the different choices of $y_0$. This is due to the fact that the distance between the initial conditions considered in Table \ref{diff_y0} and the mean of the ones represented in the form \eqref{ic_stat} is increasing, leading to the necessity of increasing the basis size to obtain the same accuracy. 
\begin{table}[htbp]
\centering
\begin{tabular}{c|cc|cc}     
      & \multicolumn{2}{c|}{SPOD} & \multicolumn{2}{c}{1 POD} \\
$y_0$ & $\mathcal{E}_{y}$ & $\mathcal{E}_{J}$ & $\mathcal{E}_{y}$ & $\mathcal{E}_{J}$ \\ \hline

0.1                     & 7.66e-5 & 3.14e-6 & 1.34e-4 & 3.11e-5 \\
0.2                     & 6.07e-5 & 1.81e-5 & 1.88e-4 & 1.52e-4 \\
$\cos(\pi \xi_1)\cos(\pi \xi_2)$& 4.72e-3 & 9.48e-3 & 1.87e-2 & 0.20    \\

 \end{tabular}
 \caption{Errors in the cost and controlled state started from different initial conditions for the 1 POD and SPOD (with $N=20$) methods, both building reduced bases of the same size $\ell=30$.}
 \label{diff_y0}
\end{table}

This demonstrates the higher accuracy of the statistical approach even for initial conditions far from the set of training samples.

\subsection{Faster control synthesis using a TT approximation}

We now assess the pre-computation of the TT approximation of the control in the feedback form, $u_{red,TT}(x^\ell) \approx u_{red}(x^\ell)$, aiming at faster synthesis of $u_{red,TT}(x^\ell)$ in the online regime by simply interpolating the TT decomposition, in contrast to computing $u_{red}(x^\ell)$ via optimisation.
Recall that as long as $x^\ell$ is fixed, the rest of the model is independent of $\mu$, therefore, $u_{red,TT}(x^\ell;\mu)$ is actually just $u_{red,TT}(x^\ell)$.
We fix the reduced model order $\ell = 20$ for all methods in this subsection.
Recall that $u_{red,TT}(x^\ell)$ is computed in the offline stage via the TT-Cross algorithm sampling $u_{red}(x^\ell)$ at certain states $x^\ell \in [-1,1]^{20}$. We fix the stopping $tol = 10^{-3}$ for the TT-Cross, and $n=6$ Legendre basis functions for discretizing each component of $x^\ell$. 
For faster computations in this offline stage, we compute $u_{red}(x^\ell)$ via the State-Dependent Riccati Equation (SDRE) applied to the semilinear system \eqref{burg_red_semi} instead of PMP.
This introduces a negligible difference to the total cost, $e.g.$ choosing $y_0 = 0.5\cos(\pi \xi_1) \cos(\pi \xi_2)$ the error between SDRE and PMP in the computation of the total cost is order $10^{-3}$.

In Table~\ref{table_Burgers_stable_tt} we notice that the SPOD method also gives a more structured reduced model which lends itself to a more accurate TT approximation of the optimal control.
\begin{table}[hbht]
\centering
\begin{tabular}{c|cc}     
  $y_0$    & SPOD &  1 POD \\ \hline
0.05   &  2.47e-6 & 1.90e-3   \\
0.1   &   7.69e-5  & 5.42e-3 \\
$0.5\cos(\pi \xi_1) \cos(\pi \xi_2)$  & 6.64e-6& 9.21e-3 \\
 \end{tabular}
  \caption{Average errors \eqref{eq:err-tt} due to the TT approximation for the 1 POD and SPOD (computed from $N=20$ trajectories) methods.}
 \label{table_Burgers_stable_tt}
\end{table}

Now we compare the SDRE controller (with and without TT approximation) and the LQR controller.
The advantage of the latter is that only one Riccati equation \eqref{are} (that for the linearization at the origin) needs to be solved and used in the computation of the feedback law \eqref{control_sdre_feed}, but we show that the absence of nonlinear terms in the control strategy leads to a slower stabilization of the system.
In the left panel of Figure \ref{sdre_lqr_cos} we show the shape of the initial condition $y_0 = 0.5\cos(\pi x) \cos(\pi y)$, and in the right panel we show the running cost of the solution with the three controllers.
We see that the LQR controller gives a higher cost, which makes it less attractive despite the computational simplicity.
The reduction of the computing time can be achieved using the TT interpolation instead.

In the same figure we see that the costs using the direct SDRE computation and the TT interpolation are indistinguishable, indicating a negligible error due to the TT approximation. 

\begin{figure}[htbp]	
\centering
\includegraphics[width=0.49\textwidth]{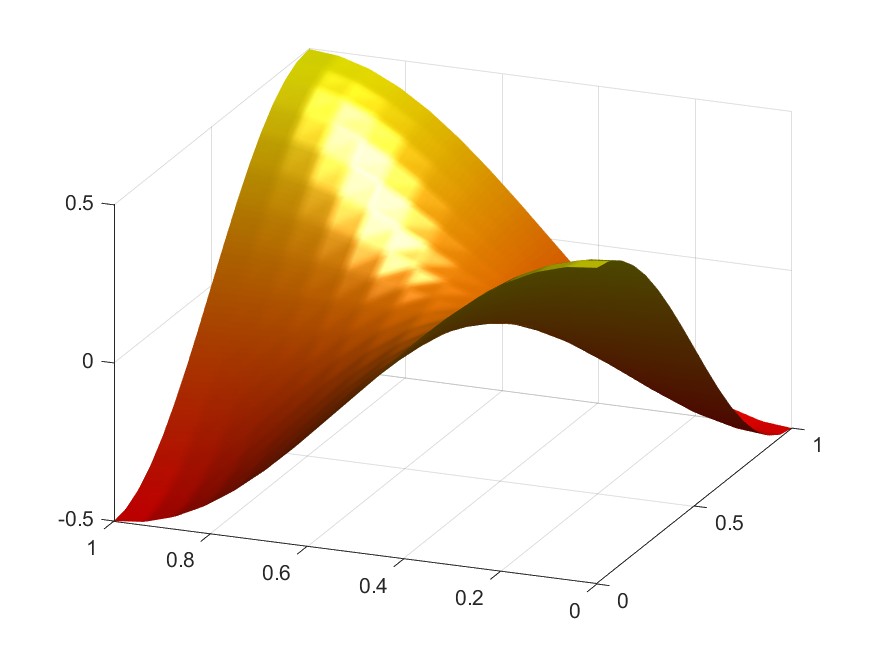}
	\includegraphics[width=0.49\textwidth]{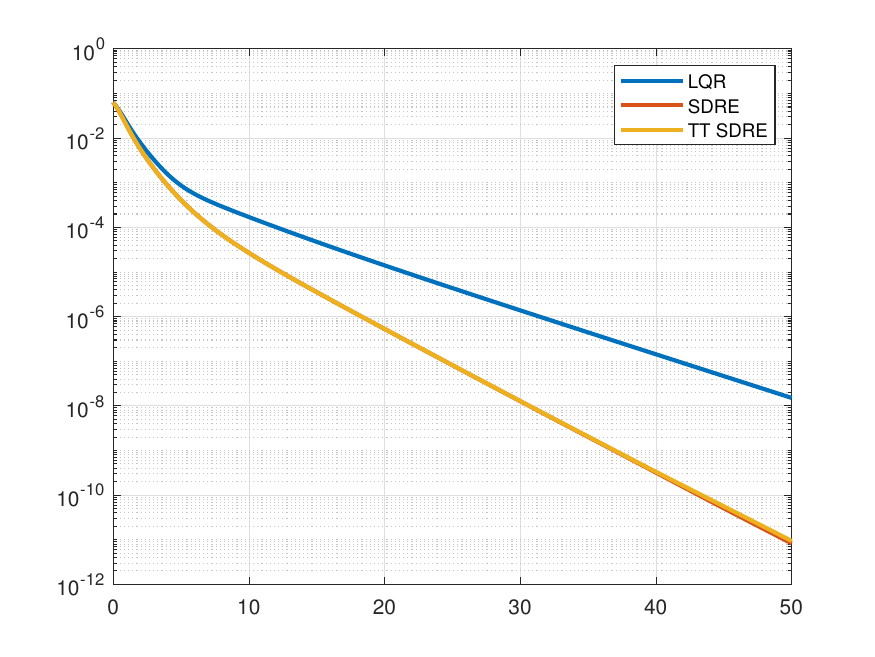}
\caption{Initial condition $y_0 = 0.5\cos(\pi \xi_1) \cos(\pi \xi_2)$ (left) and decay of the running cost with different techniques starting from this $y_0$ (right).}	
\label{sdre_lqr_cos}
\end{figure}

CPU times of computing the control for one system state are compared in Table \ref{table_Burgers_stable_cpu}.
The dimension of the reduced space is still fixed equal to $20$. First of all, we note that the application of SPOD enables a speed-up of 50 times for Pontryagin's, while the use of SDRE achieves a speed up of two orders with respect to the reduced PMP. Finally, the computation of a TT surrogate function gains a further extra speed-up order of magnitude, achieving in the end a final acceleration of 5 orders between the reduced TT-SDRE and the full PMP.

\begin{table}[hbht]
\centering
\begin{tabular}{c|c|c|c}     
  Full PMP    & Reduced PMP &  Reduced SDRE & Reduced TT \\ \hline
  41.22 s & 0.78 s &  4.25e-3 s &    4.57e-4 s
 \\  \hline
 \end{tabular}
  \caption{Comparison in terms of the CPU time for the computation of a single control with different FOM and ROM solver, fixing $\ell = 20$ and $y_0(\xi) = 0.5\cos(\pi \xi_1) \cos(\pi \xi_2)$.}
 \label{table_Burgers_stable_cpu}
\end{table}

\subsection{SPOD accuracy in the unstable case ($\alpha>0$)}
In this subsection we move to the unstable case, fixing the parameter $\alpha = 0.2$ in the dynamical system \eqref{burgers}. The addition of this term changes drastically the behaviour of the uncontrolled dynamics, characterized now by an exponential growth. It is clear that in this case the construction of the POD basis based on the uncontrolled solution is meaningless, since it quickly diverges far away from the desired state. 
In the left panel of Figure \ref{decay_sv2} we show the decay of the singular values of different POD strategies varying the number of sampled initial conditions. The decay of the singular values 

is slower than that for the stable dynamics, reflecting the more complex nature of the problem. In the right panel of Figure \ref{decay_sv2} we report the mean error in the computation of the cost functional for 1 POD and the statistical approach varying the basis size.
The initial conditions are sampled again in the form \eqref{ic_stat}. The POD method based on the uncontrolled dynamics is not reported since it achieves an error of order $\approx 10^5$, demonstrating its inefficiency in this context. SPOD presents a better accuracy for all the test cases.

\begin{figure}[htbp]	
\centering
	\includegraphics[width=0.49\textwidth]{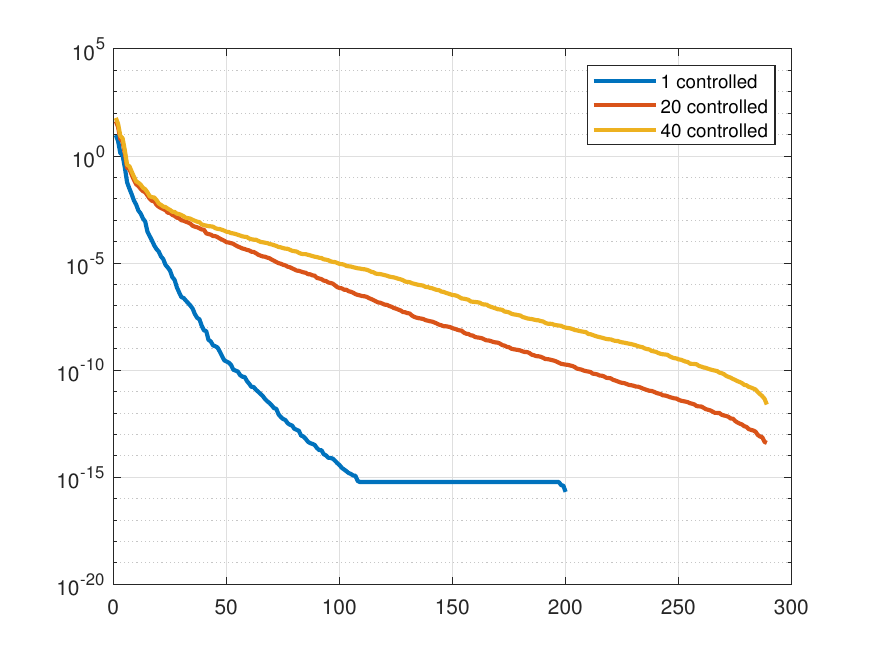}
 \includegraphics[width=0.49\textwidth]{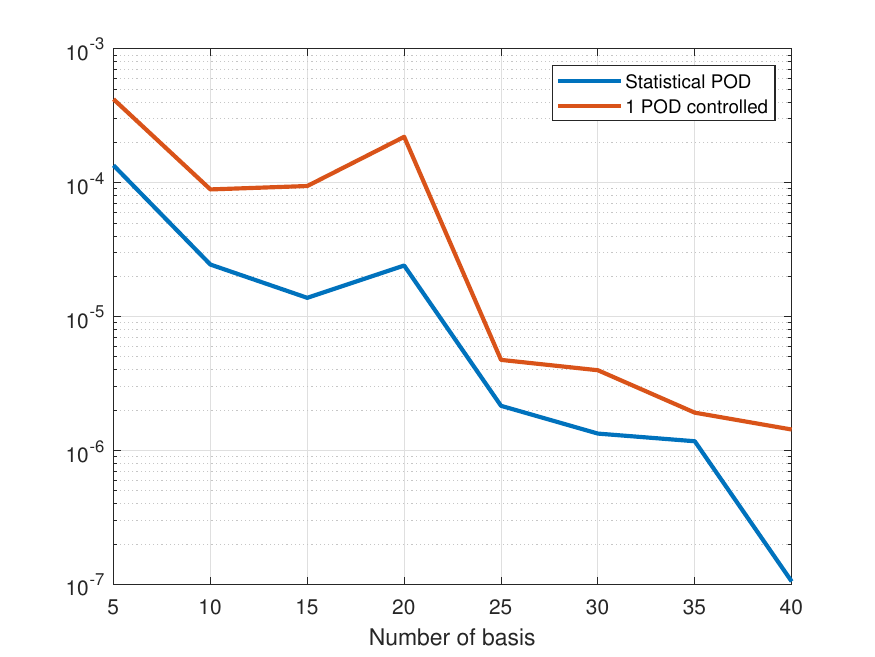}
\caption{Left: singular values of $Y_{\underline{\mu}}$ for different numbers of offline realisations $N$. Right: mean error in the cost $\mathcal{E}_{J}$ \eqref{eq:err-model} depending on the reduced basis size $\ell$ for SPOD built upon $N=40$ controlled trajectory realisations started from random iid samples from \eqref{ic_stat}, and the POD built upon 1 realisation of the controlled trajectory started from the mean $y_0$.}	
\label{decay_sv2}
\end{figure}

In Table \ref{table_Burgers2_unstable} we report the comparison of the two techniques varying the parameters appearing in the definition of the initial conditions \eqref{ic_stat} with $\ell =20$ reduced basis. Again, in the first columns we consider as random variable $\mu_{i} \sim \mathcal{N}(0, 4\sigma_0^2)$, doubling the reference standard deviation considered in \eqref{ic_stat}, while in the last columns we fix the decay exponent $\gamma=3$. The SPOD gets a better results for all the indicators, especially for the case $\sigma = 2 \sigma_0$, where $\mathcal{E}_y$ performs two order of magnitudes better than the standard POD.

\begin{table}[hbht]
\centering
\begin{tabular}{c|cc|cc}    
        & \multicolumn{2}{c|}{1 POD} & \multicolumn{2}{c}{SPOD} \\
Error   & $\sigma = 2\sigma_0$ & $\sigma=\sigma_0$ & $\sigma = 2 \sigma_0$ & $\sigma=\sigma_0$  \\
        & $\gamma=4$ & $\gamma=3$ & $\gamma=4$ & $\gamma=3$ \\\hline
$\mathcal{E}_y$     &  3.5e-4 & 3.91e-4 & 2.1e-6   &   1.33e-4\\
$\mathcal{E}_{J}$   &  7.3e-5 & 7.57e-5 & 1.8e-6   &   9.15e-6\\
 \end{tabular}
  \caption{Unstable Burgers' example: mean errors in the cost and optimal trajectory of the reduced model of order $\ell=20$ started from random iid samples from \eqref{ic_stat} varying the standard deviation $\sigma$ or the decay parameter $\gamma$. The SPOD basis is built from $N=20$ trajectories started from samples from the original \eqref{ic_stat} with $\sigma=\sigma_0$ and $\gamma=4$.
 }
 \label{table_Burgers2_unstable}
\end{table}

Finally, we pass to the application of the SDRE considering the reduced dynamical system  \eqref{burg_red_semi} and we apply the Tensor Train Cross for the construction of a surrogate model. We fix a number of basis $\ell=20$. In Table \ref{table_Burgers_unstable_tt} we compare the total cost obtained using LQR, SDRE and its approximation via TT. We immediately notice that SDRE achieves better results than LQR for all the study tests and the TT approximation obtains almost the same cost as SDRE for the digits displayed.

\begin{table}[hbht]
\centering
\begin{tabular}{c|ccc}     
  $y_0$    & LQR & SDRE &  TT-SDRE \\ \hline
0.05   &  0.0809 &  0.0808 & 0.0808   \\
 Random i.c. \eqref{ic_stat}  &    0.0798  &   0.0756 &   0.0756 \\
$0.5\cos(\pi \xi_1) \cos(\pi \xi_2)$  & 0.1592 & 0.0879 & 0.0878 \\
 \end{tabular}
  \caption{Unstable Burgers' example: total cost $J$ with different initial conditions for LQR, SDRE and TT-SDRE with SPOD (with $N=20$ samples) methods.}
 \label{table_Burgers_unstable_tt}
\end{table}

The left panel of Figure \ref{sdre_lqr_cos_unstable} shows the configuration of the controlled solution via TT-SDRE at final time with infinity norm of order $10^{-2}$. In the right panel of Figure \ref{sdre_lqr_cos_unstable} the running costs for LQR, SDRE and TT SDRE. There is a visual superposition of the curves for SDRE and TT SDRE, reflected also in Table \ref{table_Burgers_unstable_tt}. The SDRE controller is able to reduce the running cost faster, which is reflected by the lower total cost in Table \ref{table_Burgers_unstable_tt}.

\begin{figure}[htbp]	
\centering
 \includegraphics[width=0.49\textwidth]{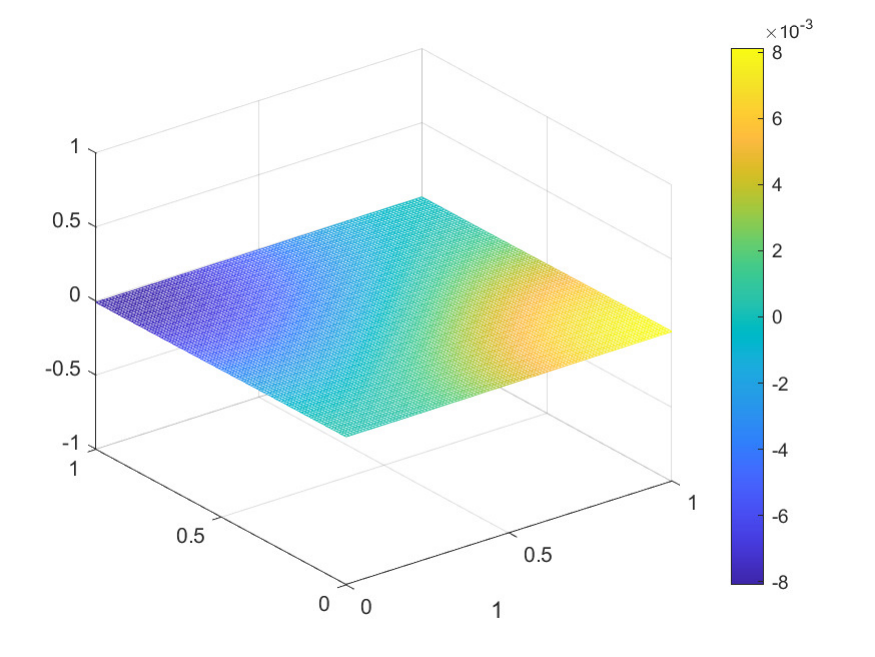}
	\includegraphics[width=0.49\textwidth]{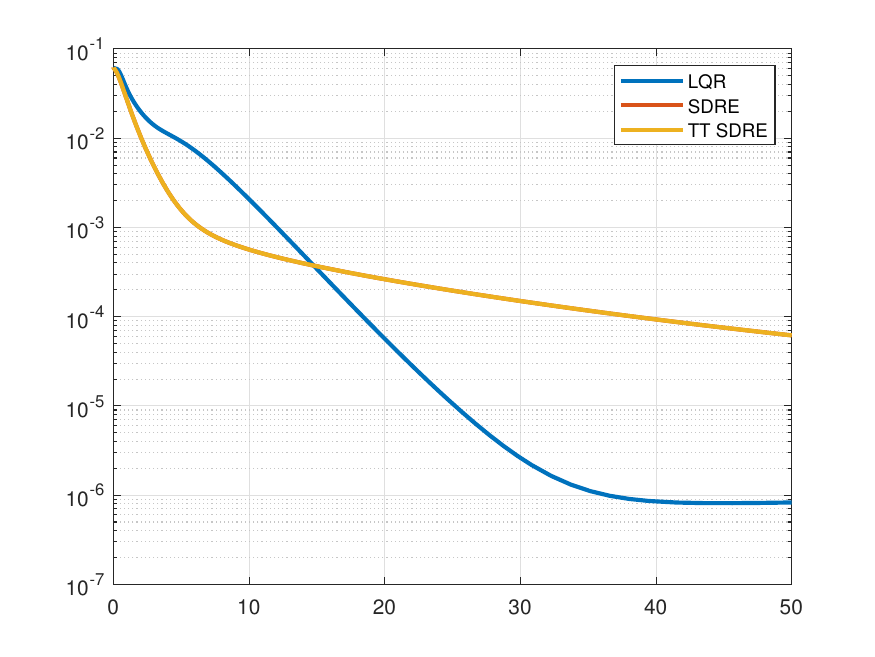}
\caption{Unstable Burgers' example: final configuration for the controlled solution via TT-SDRE (left) and decay of the running cost with different techniques (right)  starting from $y_0 = 0.5 \cos(\pi \xi_1) \cos(\pi \xi_2)$. }	
\label{sdre_lqr_cos_unstable}
\end{figure}

\section{Dirichlet boundary control for the Navier-Stokes equations}
We consider a more challenging problem: the optimal control of the 2D incompressible Navier-Stokes (NS) equation via Dirichlet boundary control.

\subsection{Problem formulation}

The NS equation reads
\begin{equation}
\begin{cases}
\partial_t y -\nu \Delta y + y \cdot \nabla y +\nabla p =0 & (\xi,t,\mu) \in \Xi \times [0,T] \times \R^M, \\
\nabla \cdot y = 0 & (\xi,t,\mu) \in \Xi \times [0,T] \times \R^M, \\
  y = g(\xi;\mu) & (\xi,t,\mu) \in \Gamma_{in} \times [0,T] \times\R^M,\\
 y = 0 & (\xi,t,\mu) \in \Gamma_{w} \times [0,T]\times\R^M,\\
  y =  u(t) & (\xi,t,\mu) \in \Gamma_{u} \times [0,T]\times\R^M,\\
  \nu \partial_n y  - p \vec{n} = 0 & (\xi,t,\mu) \in \Gamma_{out} \times [0,T]\times\R^M,\\
y(\xi,0;\mu) = y_0(\xi;\mu) &  (\xi,\mu) \in \Xi\times\R^M, \\
\end{cases}
\label{NSequation}
\end{equation}
where $\nu >0$ is the viscosity parameter and
\begin{equation}
g(\xi;\mu) = 4\xi_2(1-\xi_2) + \frac{1}{2} \sum_{k=1}^{M} k^{-\gamma} \sin(2\pi k \xi_2) \mu_k,
\label{g_stat}
\end{equation}
is the uncertain inflow with $\mu = (\mu_1,\ldots,\mu_{M})$ being independent random variables each distributed uniformly on $[-c,c]$, where the half-width $c>0$ will be varied in the numerical tests. Note that the mean inflow is given by
\begin{equation}
    \overline{g}(\xi) = 4\xi_2(1-\xi_2).
    \label{mean_inflow}
\end{equation}
The initial condition is set as
$$
y_0(\xi;\mu) = \begin{cases}g(\xi;\mu), & \xi \in \Gamma_{in}, \\ 0, & \mbox{otherwise.} \end{cases}
$$
The equation is posed on a backward facing step domain, illustrated in Figure~\ref{fig:backstep}.

\begin{figure}[t!]
\centering
\resizebox{\linewidth}{!}{

\begin{tikzpicture}{\small
 \draw[thick,blue] (-1,0) -- (-1,1) node[left,midway] {$\Gamma_{in}$}; 
 \draw[ thick] (-1,1) -- (5,1) node[above,midway] {$\Gamma_{w}$}; 
 \draw[ thick] (-1,0) -- (0,0) node[below,near start] {$\Gamma_{w}$}; 
 \draw[ thick,red] (0,0) -- (0,-1) node[right,midway] {$\Gamma_{u}$};  
 \draw[ thick] (0,-1) -- (5,-1) node[below,midway] {$\Gamma_{w}$}; 
 \draw[thick,dashed] (5,-1) -- (5,1) node[right,midway] {$\Gamma_{out}$}; 

 Inflow arrows g=4 y (1-y)
 \draw[->,blue] (-1,0.125) -- (-0.5625,0.125);
 \draw[->,blue] (-1,0.250) -- (-0.2500,0.250);
 \draw[->,blue] (-1,0.375) -- (-0.0625,0.375);
 \draw[->,blue] (-1,0.500) -- ( 0.0000,0.500) node[pos=0.5,anchor=center] {$g$};
 \draw[->,blue] (-1,0.625) -- (-0.0625,0.625);
 \draw[->,blue] (-1,0.750) -- (-0.2500,0.750);
 \draw[->,blue] (-1,0.875) -- (-0.5625,0.875);

 Corners
\node[anchor=south] at (-1,1.02) {$(-1,1)$};
\node[anchor=south] at (5,1.02) {$(5,1)$};
\node[anchor=north] at (0,-1.02) {$(0,-1)$};
\node[anchor=west] at (-0.05,0.05) {$(0,0)$};
}
\end{tikzpicture}
}
\caption{The backward step domain for the Navier-Stokes flow with uncertain inflow}
\label{fig:backstep}
\end{figure}
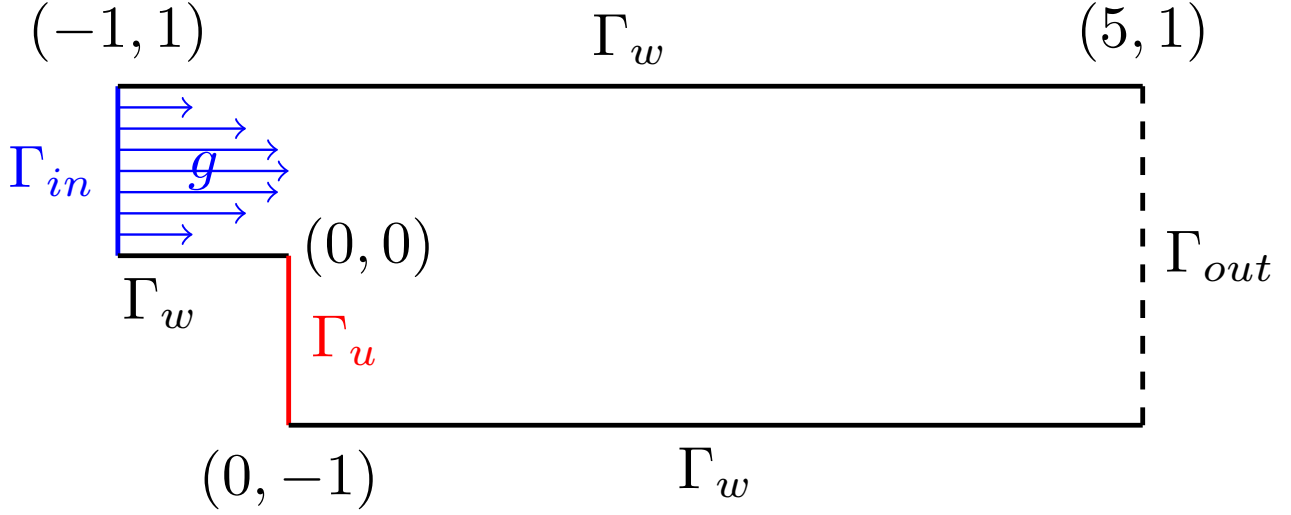

The control $u(t)$ is taken in the piecewise-constant form
$$
u(t) = \sum_{i=0}^{n_t} \begin{bmatrix}
    u_1(t_i) \\ u_2(t_i) 
\end{bmatrix} \chi_{[t_i,t_{i+1})}(t),
$$
where $\{t_i\}_{i=0}^{n_t}$ is a uniform discretization of the time interval $[0,T]$.

We introduce the following cost functional
\begin{equation}
    J_T(y,u;\mu) = \int_0^T \int_\Xi |\nabla \times y (\xi,t,\mu)|^2 \,d \xi \, dt + \int_0^{T}\delta |u(t)|^2 \; dt
    \label{cost_NS}
\end{equation}
with the aim of minimizing the vorticity of the flow over a time interval $[0,T]$ with a penalty cost for the control weighted by the parameter $\delta >0$.

Since both the control and the random field appear in the boundary conditions, it is convenient to split the solution in the form
\begin{equation}
y = \widetilde{y} +\underline{y},
\label{sum_form}
\end{equation}
where $\widetilde{y}$ takes into account the boundary conditions, and $\underline{y}$ has homogeneous boundary conditions, and formulate the feedback control problem on $\underline{y}$.
In turn, $\widetilde{y}$ needs to satisfy only the boundary and divergence-free conditions, so we choose it as the solution of the stationary Stokes equation 
\begin{equation}
\begin{cases}
 -\nu \Delta \widetilde{y}  +\nabla \widetilde{p} =0 & (\xi,\mu) \in \Xi  \times \R^M, \\
\nabla \cdot \widetilde{y} = 0 & (\xi,\mu) \in \Xi  \times\R^M, \\
 \widetilde{y} = g(\xi;\mu) & (\xi,\mu) \in \Gamma_{in} \times \R^M,\\
 \widetilde{y} = 0 & (\xi,\mu) \in \Gamma_{w} \times \R^M,\\
  \widetilde{y} =  u & (\xi,\mu) \in \Gamma_{u}  \times \R^M,\\
  \nu \partial_n \widetilde{y}  - p \vec{n} = 0 & (\xi,\mu) \in \Gamma_{out} \times \R^M.\\
\end{cases}
\label{Sequation}
\end{equation}
Note that $\widetilde{y}$ is a linear function of both $u$ and $g(\xi;\mu)$, which is in turn a linear function of $\mu$ due to \eqref{g_stat}.
This allows us to use superposition and write the Stokes solution in the form of a linear map,
\begin{equation}
\widetilde{y}(\widetilde{\mu},u) = \widetilde{Y} \widetilde{\mu} + U u, \qquad \widetilde{Y} = \begin{bmatrix}\widetilde{Y}_0 & \ldots & \widetilde{Y}_{M}\end{bmatrix}, \quad U = \begin{bmatrix}U_1 & U_2\end{bmatrix},
\label{eq:ytildemap}
\end{equation}
where $\widetilde{Y}_0$ is the solution of \eqref{Sequation} with $\mu=0$ and $u=0$ (that is, the uncontrolled Stokes solution with the mean inflow), $\widetilde{Y}_j$ for $j=1,\ldots,M$ is the solution of \eqref{Sequation} with $\mu_{j}=1$, $\mu_i=0$ for $i\neq j$ and $u=0$,
$$
\widetilde{\mu} = \begin{bmatrix}1 & \mu_1 & \cdots & \mu_{M}\end{bmatrix}^\top
$$ 
is the random vector augmented by $1$ for brevity of the map \eqref{eq:ytildemap}, and 
$U_j$ is the solution of \eqref{Sequation} with $g(\xi;\mu)=0$, $u_j=1$, $u_i=0$ for $i\neq j$.
As a by-product, this allows us to precompute \eqref{Sequation} without time dependence of $u(t)$, only for those $M+3$ initial inputs.

Plugging \eqref{eq:ytildemap} into \eqref{NSequation} we can write the following Navier-Stokes-type equation on $\underline{y}$:
\begin{equation}
\begin{cases}
\partial_t \underline{y} -\nu \Delta \underline{y} + \widetilde{y} \cdot \nabla \underline{y} + \underline{y} \cdot \nabla \widetilde{y} + \underline{y} \cdot \nabla \underline{y} + \nabla \underline{p} = \nu\Delta\widetilde{y} - \widetilde{y} \cdot \nabla \widetilde{y}, \\
\nabla \cdot \underline{y} = 0,
\end{cases}
\label{eq:NShom}
\end{equation}
with all boundary and initial conditions homogeneous.

\subsection{Semi-discretization}

We discretize \eqref{Sequation} and \eqref{eq:NShom} in space using the stable $P_2-P_1$ Taylor-Hood finite elements pair, involving bilinear elements $\{\varphi_i\}_{i=1}^{d_p}$ for the pressure and biquadratic elements $\{\phi_i\}_{i=1}^{d_v}$ for the velocity, duplicated such that $i=1,\ldots,d_v/2$ indexes the first component of the velocity, and $i=d_v/2+1,\ldots,d_v$ indexes the second component. By convention, $\phi_i$ and $\phi_{j}$ for any $i\le d_v/2$ and $j>d_v/2$ are assumed non-overlapping. The resulting semidiscretization of \eqref{eq:NShom} reads
\begin{equation}
\begin{cases}
E \dot{\underline{y}}(t) + A \underline{y}(t) + F(\widetilde{y}) \underline{y}(t) + F^*(\widetilde{y}) \underline{y}(t) + F(\underline{y}) \underline{y}(t) + D^{\top} p(t) = f_0(t), \\
D \underline{y}(t) = 0 ,
\end{cases}
\label{semiDiscreteNS}
\end{equation}
where $E,A,F(\cdot),F^*(\cdot)$ and $D$ are mass and stiffness matrices with elements
\begin{align*}
(E)_{i,j} & = \int_D \phi_i \phi_j \, d\xi, & i,j &=1, \ldots, d_v, \\
(A)_{i,j} & = \int_D \nu \nabla \phi_i \cdot \nabla \phi_j \, d\xi, & i,j &=1, \ldots, d_v, \\
(F(v))_{i,j} & = \int_D \phi_i(v \cdot \nabla) \phi_j \, d\xi, &  i,j &=1, \ldots, d_v, \\
(F^*(v))_{i,j} & = \int_D \phi_i \phi_j \partial_{\xi_k} v \, d\xi, &  i,j &=1, \ldots, d_v, \; k = \begin{cases}1, & j \le d_v/2, \\ 2, & j>d_v/2,\end{cases} \\
(D)_{i,j} & = (D_{1})_{i,j} + (D_{2})_{i,j}, && \mbox{where} \\
(D_{k})_{i,j} & = \int_D \varphi_i \partial_{\xi_k} \phi_j \, dx, & k&=1,2, \;i=1, \ldots, d_p, \;j =1, \ldots, d_v, \\
f_0(t) & = -A \widetilde{y}(\widetilde{\mu},u(t)) - F(\widetilde{y})\widetilde{y}.
\end{align*}
Note that $\dot{\widetilde{y}} = 0$ almost everywhere since $u(t)$ is piecewise-constant.
To obtain a feedback form, we can further separate $\widetilde{\mu}$ and $u$:
\begin{equation}
\begin{cases}
E \dot{\underline{y}}(t) + A \underline{y}(t) + C(\widetilde{\mu}) \underline{y}(t) \\
+ F_{yu}(\underline{y} \otimes u) 
+ F(\underline{y}) \underline{y}(t) + D^{\top} p(t) = f(\widetilde{\mu}) - B(\widetilde{\mu})u - F_{uu}(u \otimes u), \\
D \underline{y}(t) = 0 ,
\end{cases}
\label{eq:NS-mu-u}
\end{equation}
where
\begin{align*}
C(\widetilde{\mu}) & = \sum_{j=0}^{M} \widetilde{\mu}_j (F(\widetilde{Y}_j) + F^*(\widetilde{Y}_j)), \\
(F_{yu})_{i,k+2(j-1)} & = (F(U_k))_{i,j} + (F^*(U_k))_{i,j}, & i,j & = 1,\ldots,d_v, \; k=1,2,\\
(B(\widetilde{\mu}))_{i,k} & = (AU)_{i,k} + \sum_{j=0}^{M} \widetilde{\mu}_j (F(\widetilde{Y}_j)U_k + F(U_k)\widetilde{Y}_j)_i, & i&=1,\ldots,d_v, \; k=1,2,\\
(F_{uu})_{i,j+2(k-1)} & = (F(U_j)U_k)_i, & i&=1,\ldots,d_v, \; j,k=1,2,\\
f(\widetilde{\mu}) & = -A \widetilde{Y} \widetilde\mu - \sum_{j=0}^{M} \widetilde{\mu}_j F(\widetilde{Y}_j) \widetilde{Y} \widetilde{\mu}. \\
\end{align*}

Applying the same semidiscretization to the cost functional, we obtain
\begin{align}\label{disc_cost_NS}
\tilde{J}_T(y,u) & = \int_0^T  y(t)^\top D_{vort} y(t) +\delta |u(t)|^2 \; dt \\
& = \int_0^T \underline{y}(t)^\top D_{vort} \underline{y}(t) + u(t)^\top (\delta  I_2 + U^T D_{vort} U) u(t) dt \nonumber \\
& + \int_0^T 2 \widetilde{\mu}^\top \widetilde{Y}^\top D_{vort} \underline{y}(t) + 2 u(t)^\top U^\top D_{vort} \underline{y}(t) dt, \nonumber
\end{align}
where $I_2 \in \mathbb{R}^{2 \times 2}$ is the identity matrix and
$$
D_{vort} = (D_{2}-D_{1})^\top (D_{2} -D_{1}).
$$

The system of ODEs \eqref{eq:NS-mu-u}
is approximated via an implicit Euler scheme with $n_t$ time steps, and the resulting nonlinear algebraic system is solved via a Newton's method with stopping tolerance $10^{-6}$.

By default, we fix $M=8$, $\gamma=3$, $\nu=2\cdot 10^{-3}$, $\delta=10^{-3}$, $T=20$, $n_t=80$, $d_v=10382$, $d_p=1340$ obtaining a problem of dimension $d=11722$. 

\subsection{Model Order Reduction}

Next, we pass to the construction of the reduced basis and the corresponding reduced dynamical system.
Fixing a realization $\mu$, the corresponding snapshot matrix $Y_{\mu}$ will be formed just from the divergence-free velocity snapshots $\{\underline{y}^0,\ldots, \underline{y}^{n_t}\}$ with homogeneous boundary conditions. Since the solution of the ROM dynamics is a linear combination of the snapshots, the reduced trajectory directly benefits from the divergence-free property and homogeneous boundary conditions and it solves following reduced ODEs system
\begin{equation}
\begin{cases}
E^\ell \dot{y}^\ell(t) + A^\ell y^\ell(t)+ C^\ell(\widetilde{\mu}) y^\ell(t) \\ +  F^\ell_{yu}( y^\ell(t) \otimes u) +  F^\ell (y^\ell(t)) y^\ell(t) = f^\ell({\widetilde{\mu}})  -B^\ell(\widetilde{\mu})u-F^\ell_{uu}(u \otimes u), \\
y^\ell(0) = x^\ell,
\end{cases}
\label{red_semiDiscreteNS}
\end{equation}
where
\begin{align*}
E^\ell & = (U^\ell)^\top E U^\ell, \\
A^\ell & = (U^\ell)^\top A U^\ell, \\
C^\ell(\widetilde{\mu}) & =  \sum_{j=0}^{M} \widetilde{\mu}_j  \left[(U^\ell)^\top F(\widetilde{Y}_j) U^\ell + (U^\ell)^\top F^*(\widetilde{Y}_j) U^\ell\right] , \\
F^\ell_{yu} & = (U^\ell)^\top F_{yu} (U^\ell \otimes I_{2}), \\
F^\ell (y^{\ell}) & = \sum_{k=1}^{\ell} y^{\ell}_k \left[(U^\ell)^\top F(U^\ell_k) U^\ell\right],  \\
(B^\ell(\widetilde{\mu}))_{k} & = (U^{\ell})^\top A U + \sum_{j=0}^{M} \widetilde{\mu}_j \left[(U^\ell)^\top F(\widetilde{Y}_j) U_k + (U^\ell)^\top F(U_k) \widetilde{Y}_j\right], & k=1,2, \\
F_{uu}^\ell & = (U^{\ell})^\top F_{uu}, \\
f^\ell({\widetilde{\mu}}) & = -(U^{\ell})^\top A \widetilde{Y} \widetilde\mu - \sum_{j=0}^{M} \widetilde{\mu}_j (U^{\ell})^\top F(\widetilde{Y}_j) \widetilde{Y} \widetilde{\mu}, \\
\end{align*}
while the reduced cost functional reads
\begin{equation}
\label{red_cost_NS}
\tilde{J}^\ell_T(y^\ell,u) = \int_0^T y^\ell(t)^\top D^\ell_{vort} y^\ell(t) + u(t)^\top R u(t) + 2 \widetilde{\mu}^\top Y^\ell_{vort} y^\ell(t) + 2 u(t)^\top U^\ell_{vort} y^\ell(t) dt, \nonumber
\end{equation}
with
\begin{align*}
D^\ell_{vort} & = (U^\ell)^\top D_{vort} U^\ell, \\
R & = \delta I_2 + U^T D_{vort} U, \\
Y^\ell_{vort} & =  \widetilde{Y}^\top D_{vort} U^\ell \\
U^\ell_{vort} & =  U^\top D_{vort} U^\ell.
\end{align*}
Note that coefficients involving $U^{\ell}$ can be precomputed, after which the reduced ODEs system \eqref{red_semiDiscreteNS} can be both assembled and solved for each $\mu$ and $u$ with the complexity independent of the full dimension $d$.

\begin{figure}[htbp]	
\centering
	\includegraphics[width=0.49\textwidth]{decay_sing_values-eps-converted-to.pdf}
	\includegraphics[width=0.49\textwidth]{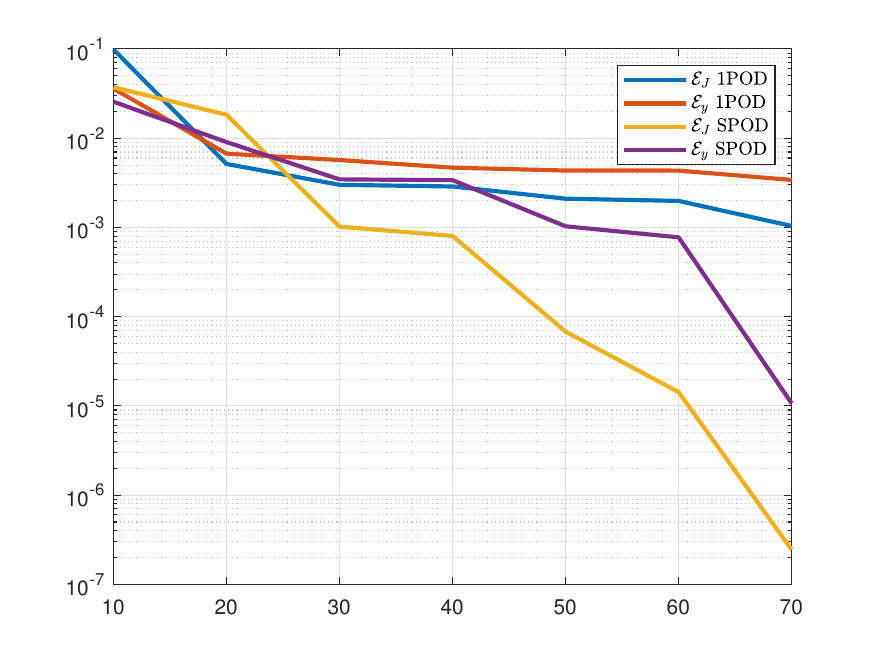}
	\caption{Left: singular values of $Y_{\underline{\mu}}$ for different numbers of offline realisations $N$. Right: model reduction errors \eqref{eq:err-model} tested on the mean inflow \eqref{mean_inflow}. The SPOD is built upon $N=15$ controlled trajectory realisations considering random iid inflows from \eqref{g_stat}, and the 1 POD is built upon the controlled trajectory with mean inflow \eqref{mean_inflow}.}
 \label{decay_NS}
\end{figure}
  
\subsection{Accuracy of the controlled reduced model}
In the left panel of Figure \ref{decay_NS} we show the decay of the singular values for 1 POD and for the SPOD using different number of realizations for the collection of the snapshots ($N \in \{5,10,15\}$). 
We note that the singular values above $10^{-3}$ are almost indistinguishable in SPOD using 10 and 15 realisations.
This shows that $N=15$ is sufficient for the statistical procedure to converge for the most relevant modes. 
We now pass to study the performances of the SPOD approach and its difference with the 1 POD technique. The statistical approach has been constructed upon $N=15$ trajectories controlled by PMP considering the random inflow \eqref{g_stat} and fixing the parameter $c=1$. The right panel of Figure \ref{decay_NS} displays the behaviour of the different error indicators for a system where the inflow is set to its mean value, $g = \overline{g}(\xi)$. It is interesting to note that although 1 POD is built exactly upon the controlled snapshots of the mean inflow, leading to a projection error of order $10^{-5}$ for $\ell =70$, the SPOD performs better for both error indicators for higher reduced dimensions, reflecting the richness of the SPOD basis for accommodating different trajectories. 

\begin{figure}[htbp]	
\centering
	\includegraphics[width=0.49\textwidth]{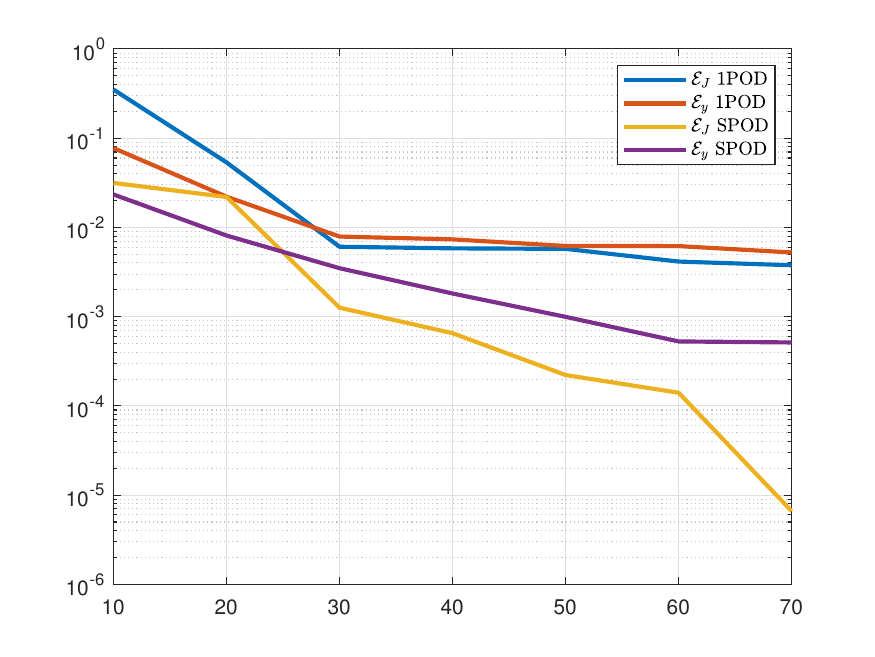}
 \includegraphics[width=0.49\textwidth]{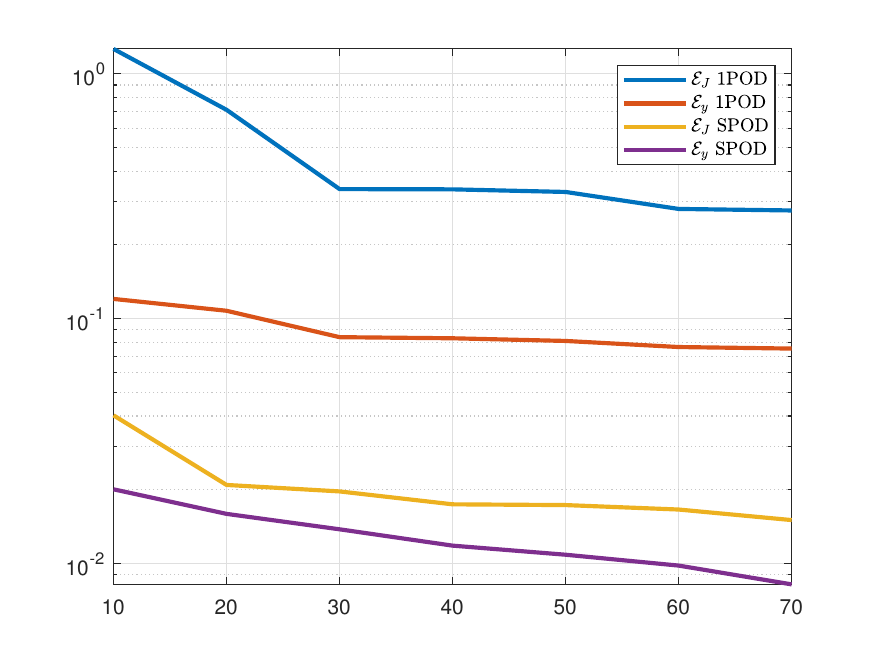}
	\caption{Average errors over 10 iid random inflows with $c=1$ (left) and $c=2$ (right) versus the reduced basis size $\ell$. The SPOD is built upon $N=15$ controlled trajectory realisations with random iid inflows from \eqref{g_stat}. The 1 POD is built upon 1 realisation of the controlled trajectory with mean inflow $\overline{g}(\xi)$.}	
  \label{err_c12_NS}
\end{figure}

In Figure \ref{err_c12_NS} we show the analysis of the mean of the errors \eqref{eq:err-model} on 10 random samples of $\mu$ changing the half-width $c$ of the interval of $\mu_k$. In the left panel we consider $c=1$, employed also for the construction of the statistical basis. The results are immediately evident: all the error indicators for the 1 POD resolutions are stuck between order $10^{-2}$ and $10^{-3}$, while SPOD presents in general a decreasing behaviour, reaching order $10^{-5}$ for the error indicator $\mathcal{E}_{J}$ with 70 reduced basis vectors. In the right panel the numerical experiments are run with parameter $c=2$, introducing  optimal trajectories possibly far from those considered during the basis construction. This is reflected in the order of the different error indicators, which achieve at most order $10^{-2}$, but still yielding a better approximation compared to the 1 POD strategy.

\subsection{Suboptimal controllers}
Lastly, we compare different faster controllers: the PMP applied to the full order model but with the mean inflow,
the SDRE applied to the reduced order model using 1 POD and SPOD bases, as well as the LQR applied to reduced models. 
Firstly, we show the uncontrolled flow at the final time in Figure~\ref{uncontrolled}
with a prefixed random inflow \eqref{g_stat} 
with 
\begin{equation}\label{eq:ns_mu_fixed}
\mu_* = \begin{bmatrix} 
\;\;\; 0.0984\\ -0.3838\\ -0.1259\\  \;\;\;  0.0398\\ -0.4314\\ \;\;\;  0.1589\\ \;\;\;  0.2323 \\ \;\;\;  0.0947
\end{bmatrix} .
\end{equation}

On the left panel we show the absolute velocity $|v|(\xi)=\sqrt{y_1(\xi)^2+y_2(\xi)^2}$, while on the right panel we show the velocity vector field $y=(y_1(\xi),y_2(\xi))$. It is possible to notice the presence of vortexes due to re-circulation issues. The total cost of the uncontrolled dynamics is equal to $4.6279$.
\begin{figure}[htbp]	
\centering
\includegraphics[width=0.49\textwidth]{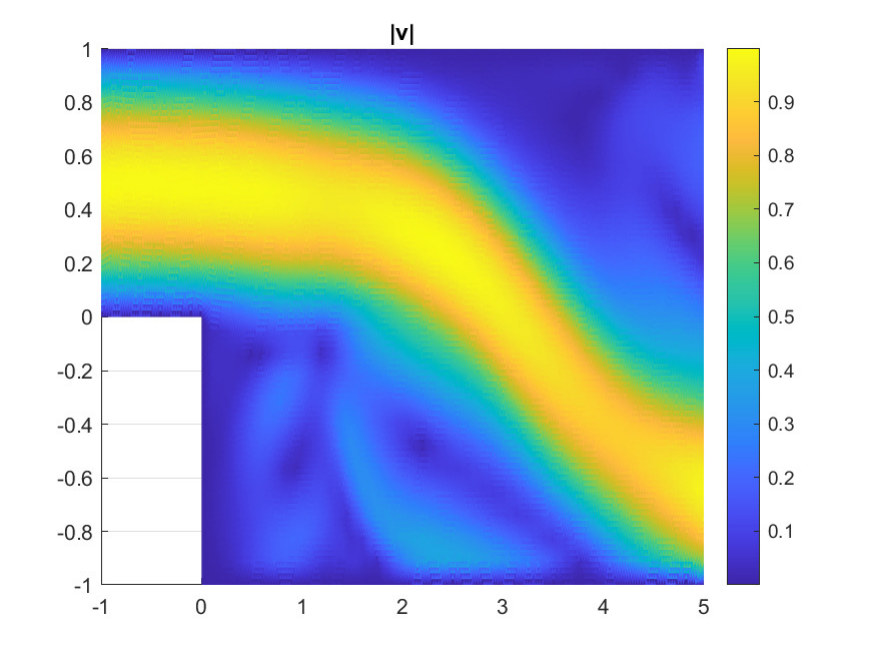}
 \includegraphics[width=0.49\textwidth]{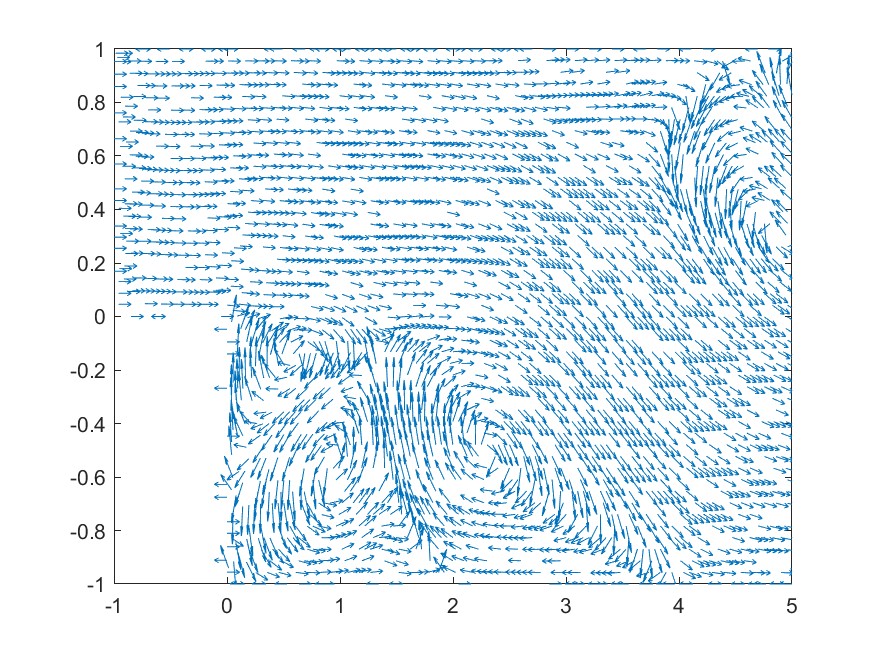}
	\caption{(Uncontrolled case) Absolute velocity (left) and velocity vector field (right) at final time $T=20$. $J_T(y,u) = 4.6279$.}
 \label{uncontrolled}
\end{figure}
Using PMP on the full system with the mean inflow $\overline{g}(\xi)$ to compute the control signal, but applying this signal to the system with a random realisation of the inflow gives the flow as shown in Figure~\ref{TT_controlled_pont_noise}.
We see that this control is unable to reduce the vortexes completely.
This indicates the need for controllers that are more specific and robust to random inputs to the model.

\begin{figure}[htbp]	
\centering
	\includegraphics[width=0.49\textwidth]{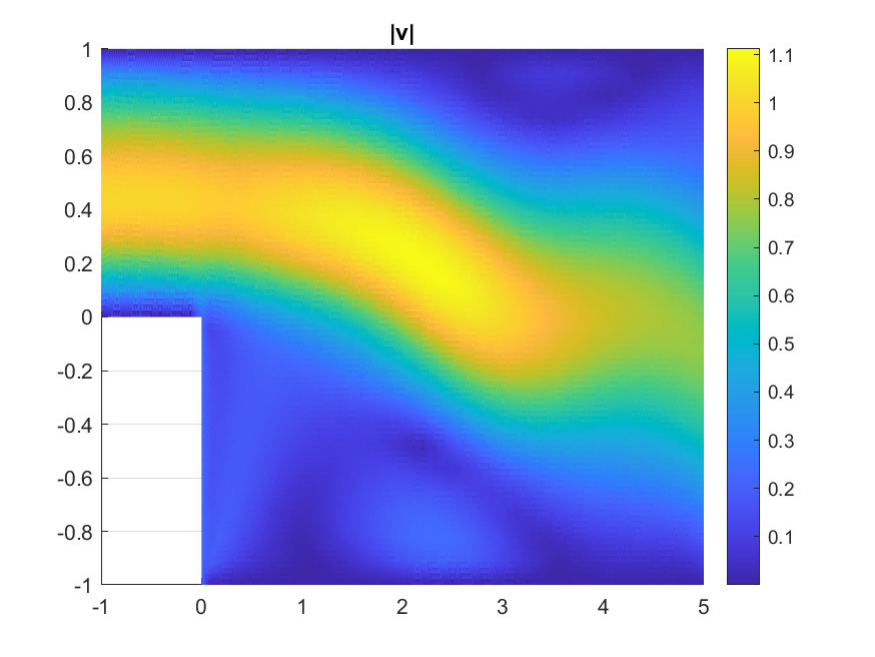}
 \includegraphics[width=0.49\textwidth]{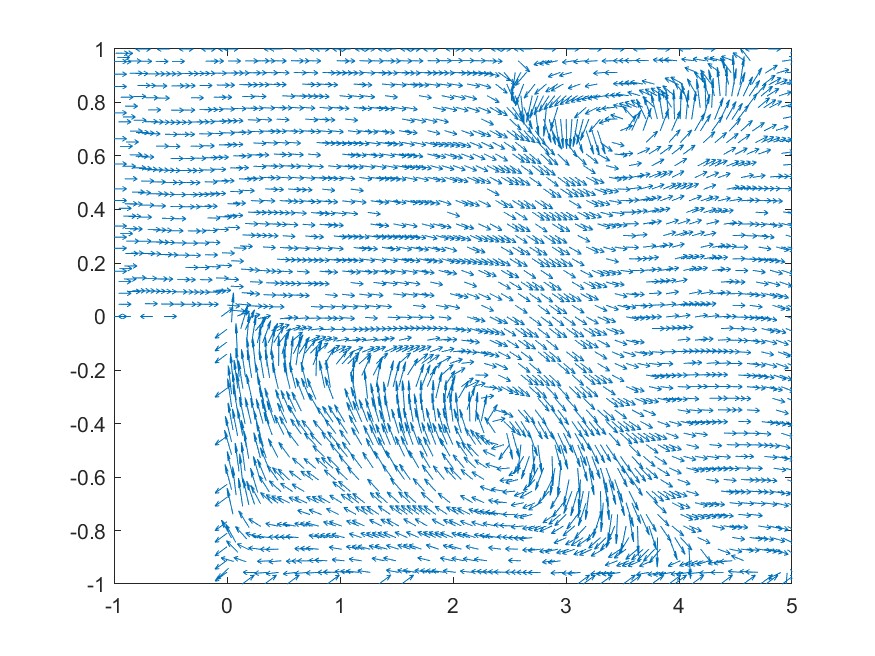}
	\caption{(Deterministic controller) Mean velocity (left) and velocity vector field (right) at final time $T=20$. The optimal control is computed via Pontryagin for the Full Order dynamics in absence of noise and plugged into a perturbed dynamical system with random inflow \eqref{g_stat}. $J_T(y,u) = 4.9015$.}	
 \label{TT_controlled_pont_noise}
\end{figure}

Now consider applying SDRE and LQR to the reduced system \eqref{red_semiDiscreteNS}. Since SDRE does not take into account quadratic terms in the control, we must omit them, and write down the reduced dynamics in a semilinear form
\begin{equation}
\dot{y}^\ell(t) = \mathcal{A}^\ell(y^\ell(t)) y^\ell(t)+ \mathcal{B}^\ell(y^\ell(t))  u (t), \quad y^\ell(0) = x^\ell,
\label{eq_NS_sdre}
\end{equation}
where
$$
\mathcal{A}(y^\ell(t)) = -(E^\ell)^{-1} (A^\ell +C^\ell(\widetilde{\mu})+F^\ell (y^\ell(t))),
$$
$$
\mathcal{B}(y^\ell(t)) =  -(E^\ell)^{-1} (F^\ell_{yu} ( y^\ell(t) \otimes I_2)+B^\ell(\widetilde{\mu})).
$$

First we compute the solution controlled via a LQR feedback, solving the Riccati equation \eqref{sdre} for $x=\underline{0}$, obtaining the matrix $P_0$.  At this point we consider the linearized feedback map
\begin{equation*}
     u(x^\ell) = -R^{-1} (\mathcal{B}(\underline{0})^\top P_0+2 U^\ell_{vort})x^\ell 
    \label{control_lqr}
\end{equation*}
and the resulting total cost is $3.3228$. The final configuration and the velocity vector field is shown in Figure \ref{lqr_NS}. We note that the resulting flow is less turbulent than the uncontrolled case, but the solution is still far from the laminar regime.

 \begin{figure}[htbp]	
\centering
\includegraphics[width=0.49\textwidth]{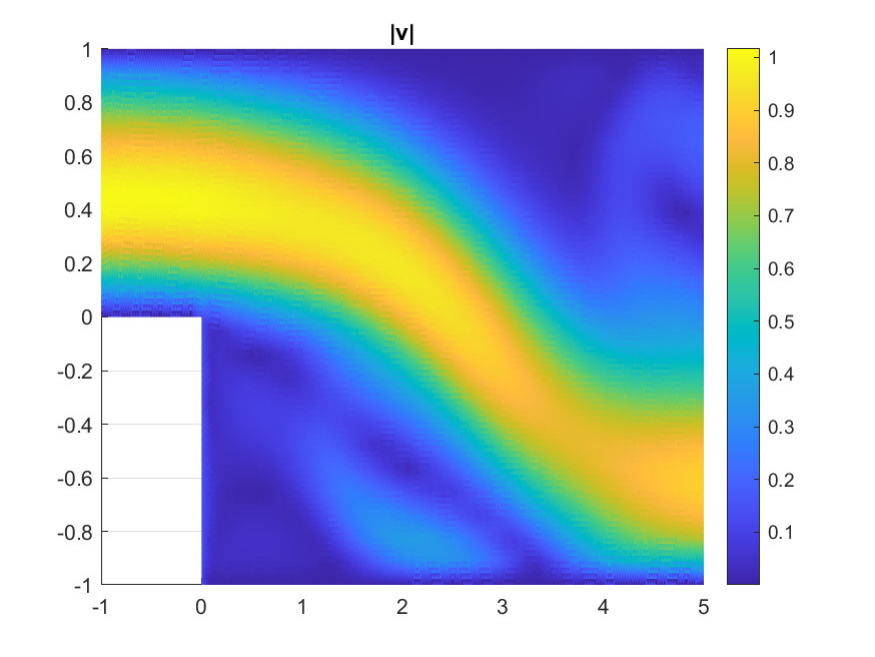}
 \includegraphics[width=0.49\textwidth]{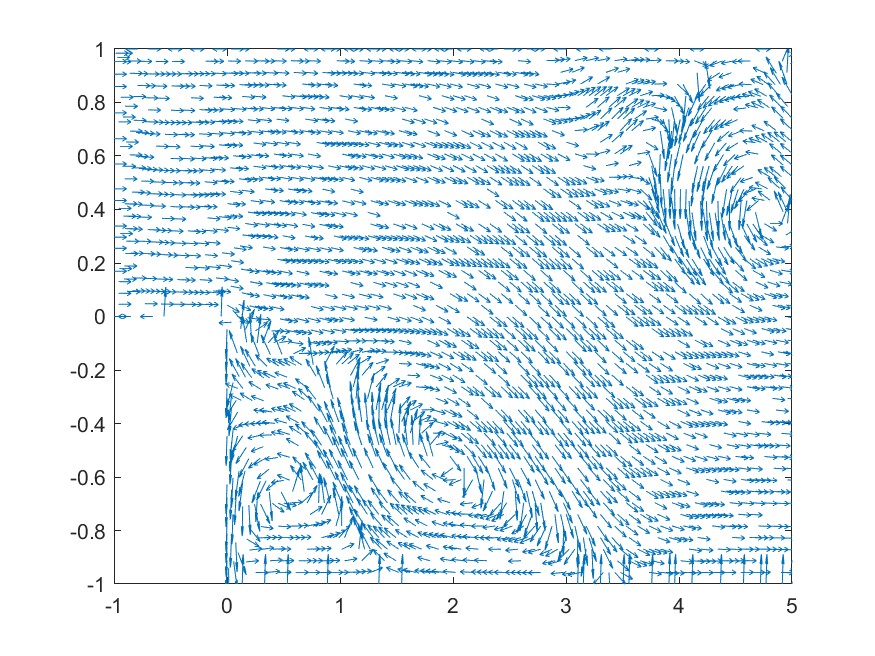}
	\caption{(LQR controller) Mean velocity (left) and velocity vector field (right) at final time $T=20$. $J_T(y,u) = 3.3228$.}	
 \label{lqr_NS}
\end{figure}

\subsection{Tensor Train approximation of the reduced SDRE control}
Now we consider the TT approximation of the feedback control function computed by SDRE on 1 POD and SPOD reduced models.
Since the TT Cross approximates scalar functions, it is applied twice, one per each component of the control \eqref{control_sdre_feed}.
Moreover, we fix $\mu=\mu_*$ as defined in \eqref{eq:ns_mu_fixed}. This allows us to approximate again a function $u(x^\ell)$ depending on the reduced state only.
In these state variables, we consider a domain $X^\ell = \bigtimes_{i=1}^\ell [a_i,b_i]$, where the interval ranges are chosen such that the domain contains all the reduced snapshots. Furthermore, we fix $tol = 10^{-3}$ and we consider $n=6$ Legendre basis functions per dimension.

Table \ref{table_tt_NS} displays the approximation error $\mathcal{E}_{TT}$ \eqref{eq:err-tt}.
The error in the statistical framework is in the order of $10^{-2}$, while the approximation error with the 1 POD basis is order $10^{-1}$, reflecting the fact that the projection onto the 1 POD basis is not able to approximate perturbed trajectories.

\begin{table}[hbht]
\centering
\begin{tabular}{c|cc}     
$\ell$  & $err_{TT}$ 1 POD & $err_{TT}$ SPOD  \\ \hline
5 & $2.49e-1$ & $3.60e-2$ \\ 
10 & $3.50e-1$ & $6.21e-3$ \\
20 & $1.19e-1$ & $1.62e-2$ \\
 \end{tabular}
  \caption{Error in the TT approximation using POD and SPOD.}
 \label{table_tt_NS}
\end{table}

Finally, in Figures \ref{TT_controlled_1POD}-\ref{TT_controlled_3} we show the final configuration of the controlled solution respectively for TT-SDRE-1POD and TT-SDRE-SPOD. For the 1 POD approach, it is possible to note by the right panel of Figure \ref{TT_controlled_1POD} that a turbulent regime is still active, while for the statistical approach (in Figure~\ref{TT_controlled_3}) the fluid presents a more laminar behaviour.

\begin{figure}[htbp]	
\centering
	\includegraphics[width=0.49\textwidth]{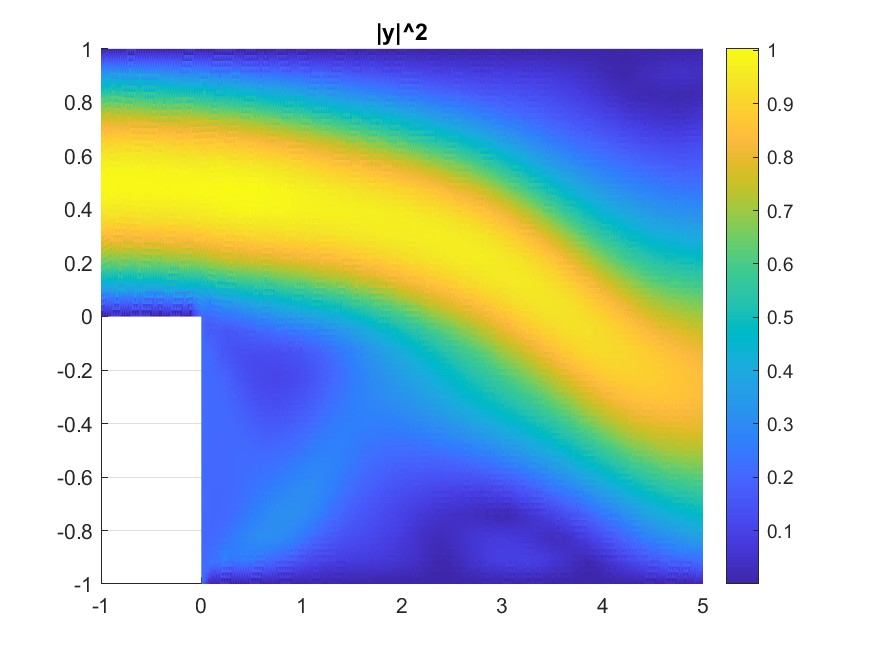}
 \includegraphics[width=0.49\textwidth]{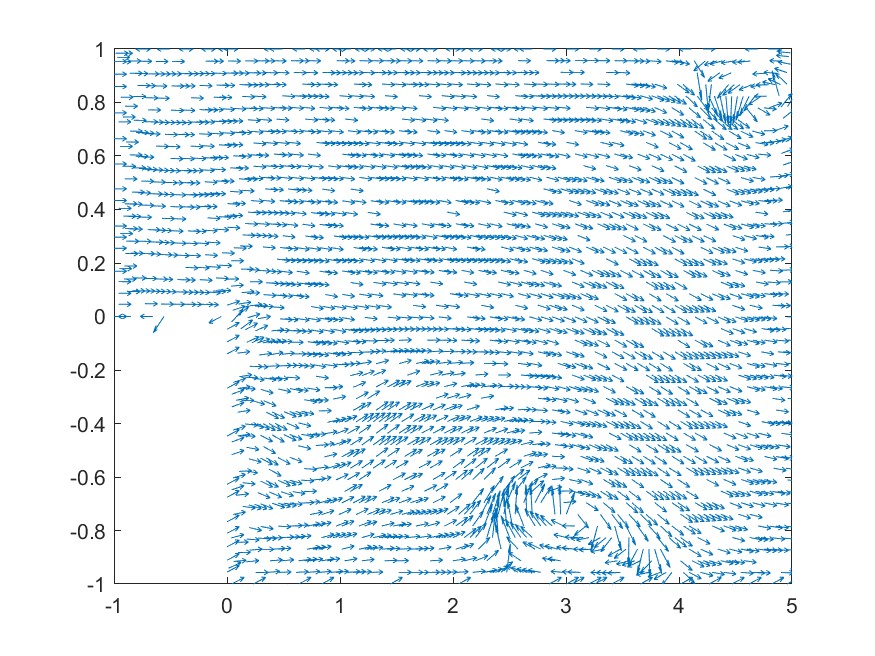}
	\caption{(1 POD controller) Mean velocity (left) and velocity vector field (right) at final time $T=20$. The optimal control is computed via Tensor Train Cross and 1 POD with $\ell =20$ basis. $J_T(y,u) = 3.0880$.}
 \label{TT_controlled_1POD}
\end{figure}

\begin{figure}[htbp]	
\centering
	\includegraphics[width=0.49\textwidth]{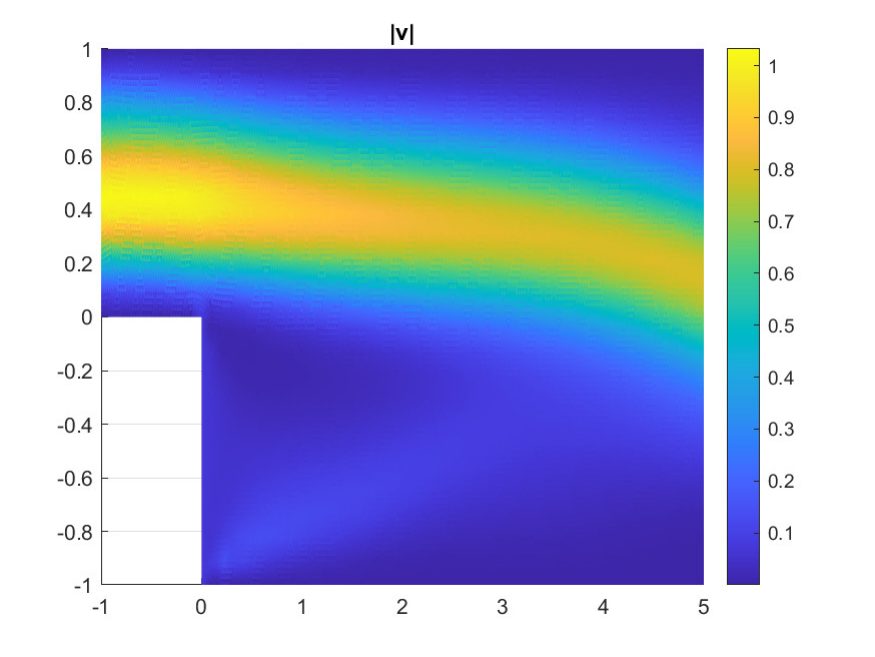}
 \includegraphics[width=0.49\textwidth]{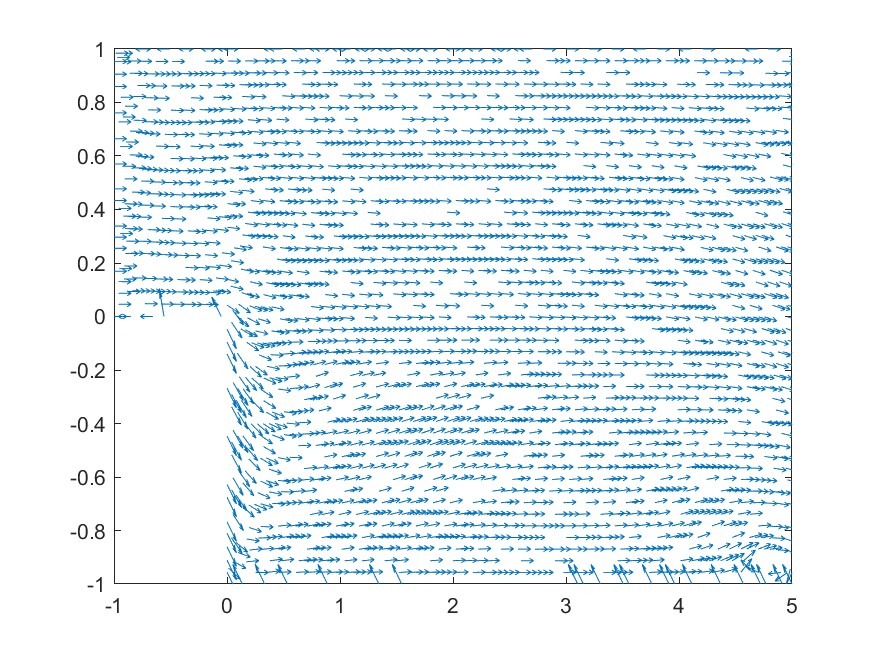}
	\caption{(SPOD controller) Mean velocity (left) and velocity vector field (right) at final time $T=20$. The optimal control is computed via Tensor Train Cross and SPOD with $\ell =20$ basis. $J_T(y,u) = 2.9527$.}	
 \label{TT_controlled_3}
\end{figure}

\section*{Concluding remarks}
We have developed a model order reduction method for synthesis of feedback control laws for nonlinear, parameter-dependent dynamics, including fluid flow problems. The reduction phase is inspired by POD techniques, requiring sampling  of the (sub)optimal control problem solutions for different realizations of the random variables. Snapshots are compressed for the construction of a statistical POD basis which minimizes the empirical risk.

The resulting reduced order model facilitates the construction of a data-driven stabilizing feedback law in the tensor train format. The low-rank tensor train structure enables the real-time implementation of a feedback control for high-dimensional problems such as vorticity minimization in 2D Navier-Stokes. 

Future research directions include the design of robust $\mathcal{H}_{\infty}$ controllers, the development of greedy sampling strategies for both parameters and initial conditions which can alleviate the computational cost of the offline phase, and the training of higher dimensional surrogates using physics-informed neural networks.

\section*{Acknowledgements}
This research was supported by the UK Engineering and Physical Sciences Research Council New Horizons Grant EP/V04771X/1, the New Investigator Award EP/T031255/1 and the Standard Grant EP/T024429/1.

\bibliographystyle{elsarticle-num}
\bibliography{references}

\end{document}